\documentclass{gsm-l}
\usepackage{amsmath,amssymb,amscd,latexsym,mathrsfs,stmaryrd,wasysym,epic}%,mathrsfs}
\usepackage{pb-diagram}
\usepackage{diagrams}
\usepackage{tikz}
\usepackage{verbatim}
\usepackage{preview}
\PreviewEnvironment{center}
\newtheorem{Theorem}{Theorem}[section]
\newtheorem{Proposition}[Theorem]{Proposition}
\newtheorem{Lemma}[Theorem]{Lemma}
\newtheorem{Corollary}[Theorem]{Corollary}

\theoremstyle{definition}
\newtheorem{Definition}[Theorem]{Definition}

\theoremstyle{remark}
\newtheorem{Remark}[Theorem]{Remark}
\newtheorem{Example}[Theorem]{Example}
\newtheorem{Exercise}[Theorem]{Exercise}

\numberwithin{section}{chapter}
\numberwithin{equation}{section}

\def\Rep{\operatorname{Rep}}

\def\mod#1{#1\!\operatorname{-mod}}
\def\Mod#1{#1\!\operatorname{-Mod}}

\def\id{\operatorname{id}}

\def\C{{\mathbb C}}

\def\Z{{\mathbb Z}}
\def\F{{\mathcal F}}

\def\Par{{\mathcal P}}
\def\spa{\operatorname{span}}
\def\triv{{\bf 1}}

\def\0{{\bar 0}}
\def\1{{\bar 1}}

\def\B{{\mathcal B}}

\def\a{{\mathfrak a}}

\def\Hom{{\operatorname{Hom}}}

\def\End{{\operatorname{End}\,}}

\def\ind{{\operatorname{ind}\,}}

\def\coind{{\operatorname{coind}}} % dont expect we need it
\def\res{{\operatorname{res}\,}}
\def\Res{{\operatorname{cont}\:}}
\def\im{{\operatorname{im}\,}}

\def\sgn{{\operatorname{\bf sgn}}}

\def\diag{{\operatorname{diag}}}

\def\equivalent{{\ \Leftrightarrow\ }}

\def\lan{\langle}
\def\ran{\rangle}

\def\bi{{\underline{i}}}
\def\bj{{\underline{j}}}
\def\bk{{\underline{k}}}

\def\eps{{\varepsilon}}
\def\phi{{\varphi}}
\def\emptyset{{\varnothing}}
\def\ga{{\gamma}}

\def\la{{\lambda}}
\def\om{{\omega}}
\def\La{{\Lambda}}
\def\si{{\sigma}}

\def\de{{\delta}}
\def\De{{\Delta}}
\def\al{{\alpha}}
\def\be{{\beta}}
\def\Om{{\Omega}}

\def\into{{\hookrightarrow}}

\def\to{\rightarrow}
\def\iso{\,\tilde\rightarrow\,}
\def\nd{{\not{\mid}}\,}

\def\MatRep#1{{\tt{MatRep}}(#1)}

\newdimen\hoogte    \hoogte=12pt    
\newdimen\breedte   \breedte=14pt  
\newdimen\dikte     \dikte=0.5pt 
\newenvironment{Young}{\begingroup
       \def\vr{\vrule height0.89\hoogte width\dikte depth 0.2\hoogte}
       \def\fbox##1{\vbox{\offinterlineskip
                    \hrule height\dikte
                    \hbox to \breedte{\vr\hfill##1\hfill\vr}
                    \hrule height\dikte}}
       \vbox\bgroup \offinterlineskip \tabskip=-\dikte \lineskip=-\dikte
            \halign\bgroup &\fbox{##\unskip}\unskip  \crcr }
       {\egroup\egroup\endgroup}
\def\diagram#1{\relax\ifmmode\vcenter{\,\begin{Young}#1\end{Young}\,}\else%
              $\vcenter{\,\begin{Young}#1\end{Young}\,}$\fi}

\begin{document}
\def\xypic{\hbox{\rm\Xy-pic}}
\frontmatter
\title{Ess\'en Lectures: Representation Theory of Symmetric Groups}
\author{\sc Alexander Kleshchev}
%\address
%{Department of Mathematics\\ University of Oregon\\
%Eugene\\ OR~97403, USA}
%\email{klesh@math.uoregon.edu}

\iffalse{
\thanks{
{2000 subject classification: 17B67, 20C08, 20C20, 17B10, 17B37.}\\
{\phantom{sp-} Both authors
partially supported by the NSF (grant nos DMS-9801442 and DMS-9900134).}
}

\begin{abstract}
Spring 2013
\end{abstract}
}\fi

\maketitle

\tableofcontents

\mainmatter

%\part{Representations of quivers}

These are partial lecture notes from the fifteen Ess\'en Lectures for graduate students  at Uppsala University given (in four days!) in June 2013.

\chapter{Day One}

\section{Group representation theory}
We begin with a general review of group representation theory. 

Fix a ground field $F$, which in these lectures will usually be the field of complex numbers $\C$. Fix also for the moment an arbitrary group $G$. 

There are several equivalent ways to think about representations of $G$ over $F$. 
Let $V$ be an $F$-vector space. 
A {\em representation} of $G$ in $V$ is a homomorphism $\rho:G\to GL(V)$. A {\em representation} of $G$ (over $F$) is a representation of $G$ in some $F$-vector space $V$. 

Given a representation $\rho:G\to GL(V)$ we can define an action of $G$ on $V$ via $gv=\rho(g)(v)$. Thus we get a notion of a $G$-module over $F$. 
Extending by linearity we get a structure of an $FG$-module on $V$, where $FG$ is the group algebra. In this way, the notion of an $F$-representation of $G$ is the same as the notion of an $FG$-module, and everything in Exercise~\ref{ExRepAlg} applies to this situation. In particular we can speak of {\em matrix representations} of finite groups which are just group homomorphisms $\rho:G\to GL_n(F)$. 

Comments on why to do representation theory. 

A left $FG$-module $V$ is called {\em simple} or {\em irreducible} if $V\neq 0$ and $V$ has no submodules different from $0$ and $V$. The main problem of representation theory is to classify irreducible modules. 

In these lectures we will be mainly concerned with representations of finite groups over $\C$. From some very general point of view, this is a trivial subject. Indeed, $\C G$ is a finite dimensional algebra over an algebraically closed field. Moreover, by Maschke's Theorem, every $\C G$-module is semisimple, i.e. the algebra $\C G$ is semisimple in the sense of Wedderburn. By the classical Wedderburn-Artin Theorem, we now must have

\begin{equation}\label{wd}
\C G \cong M_{n_1}(\C) \oplus\cdots\oplus M_{n_r}(\C). 
\end{equation}

We know that each matrix algebra $M_{n_i}(\C)$ has a unique irreducible module up to isomorphism, namely $\C^{n_i}$ with the natural action of $M_{n_i}(\C)$ on the column vectors of $\C^{n_i}$. Moreover, the irreducible module $\C^{n_i}$ lifts to an irreducible module over the direct sum $M_{n_1}(\C) \oplus\cdots\oplus M_{n_r}(\C)$ with the ``wrong'' Wedderburn components $M_{n_j}(\C)$ for $j\neq i$ acting trivially and the ``correct'' Wedderburn component $M_{n_i}(\C)$ acting as before. In this way, we get all non-isomorphic irreducible modules over the algebra $M_{n_1}(\C) \oplus\cdots\oplus M_{n_r}(\C)$ up to isomorphism. (Check this!)

Thus, $r$ is the number of isomorphism classes of irreducible $\C G$-modules
and $n_1,\dots,n_r$ are their dimensions. 
Denote the corresponding irreducible $\C G$-modules by
\begin{equation}\label{EIrrG}
L_1,\dots,L_r.
\end{equation}
Going backwards, if we know the irreducible representations $L_1,\dots,L_r$ of $\C G$, then picking linear bases in them, gives a decomposition
$$
\C G\cong M_{\dim L_1}(\C)\oplus\dots\oplus M_{\dim L_r}(\C).
$$

Let $e_i := (0,\cdots,0, I_{n_i}, 0, \cdots, 0)$ be the identity matrix of the $i$th matrix
algebra. Then 
\begin{equation}\label{EIdG}
e_1,\dots,e_r \in \C G
\end{equation}
are mutually orthogonal central idempotents 
summing to the identity. Since the center of a matrix algebra is one dimensional, spanned by the
identity matrix, $\{e_1,\dots,e_r\}$ is a basis of $Z(\C G)$. Moreover, it is clear that $Z(\C G)$ is a commutative semisimple algebra isomorphic to $\C\oplus\dots\oplus \C$ ($r$ copies).

Note that $e_i$ acts on the $j$th irreducible module $L_j$ 
as $\delta_{i,j}$. 
Considering dimension of each side of (\ref{wd}) as a $\C$-vector space we conclude: 
$$
|G| = (n_1)^2 + \cdots + (n_r)^2.
$$

The number $r$ has a convenient group theoretic interpretation:

\begin{Lemma} The number $r$ in (\ref{wd}) is equal to the
number of conjugacy classes in the group $G$.
\end{Lemma}
\begin{proof} Let us compute $\dim Z(\C G)$ in two different ways.
First, we already know that $\dim Z(\C G) = r$. 
On the other hand, if
$
\sum_{g \in G} a_g g \in \C G
$
is central then conjugating by $h \in G$ you see that
$
a_g = a_{hgh^{-1}}
$
for all $h \in G$.
Hence the coefficients $a_g$ are constant on conjugacy classes.
So if $C_1,\dots,C_s$ are the conjugacy classes of $G$, 
the elements $z_i = \sum_{g \in C_i} g$ form a basis 
$\{z_1,\dots,z_s\}$ for $Z(\C G)$. 
Hence $r = \dim Z(\C G) = s$.
\end{proof}

The elements $z_1,\dots,z_r\in Z(\C G)$ introduced in the proof of the lemma are referred to as {\em class sums}. We saw that they form a basis of $Z(\C G)$. 
%There are other important bases of $Z(\C G)$. 
The connection between this basis and the basis $\{e_1,\dots,e_r\}$ can be clarified using character theory.

By the way, from the Wedderburn decomposition it is also to see the so called Schur's lemma: 

\begin{Lemma}\label{LSchur}%{\rm \cite{}}%
{\bf (Schur's Lemma)}
Let $V$ and $W$ be irreducible $\C G$-modules. 
\begin{enumerate}
\item[{\rm (i)}] If $V\not\cong W$, then $\Hom_G(V,W)=0$.\item[{\rm (ii)}] $\End_G(V)=\C\cdot \id_V$. 
\end{enumerate}
\end{Lemma}

\begin{Example}
Let $G$ be abelian. Then there are $r = |G|$ conjugacy classes, and
$n_1^2 + \cdots  + n_r^2 = r$ hence each $n_i = 1$. 
So there are $r$ isomorphism classes of
irreducible $\C G$-module, all of which are one-dimensional. 
To construct the irreducible $\C G$-modules explicitly let us switch to the language of matrix representations, so we have to classify the group homomorphisms $\rho:G\to\C^\times$. By the Fundamental Theorem of Abelian Groups we can decompose $G=C_{a_1}\times\dots\times C_{a_m}$ as a product of cyclic groups with generators $x_1,\dots,x_m$, respectively. Let $\eps_i\in\C$ be a primitive $a_i$th root of $1$, and note that $\rho(x_i)=\eps_i^{l_i}$ for some $0\leq l_i<a_i$, $i=1,\dots,m$. Note that the choice of $l_1,\dots,l_m$ determines $\rho$ explicitly, and there are $|G|$ possible choices, so we have obtained all possible homomorphisms. 
\end{Example}

%A representation $\rho: G \rightarrow GL(V)$ is {\em faithful} if $\rho$ is injective.

\begin{Example}
For any group $G$, there is always the {\em trivial $\C G$-module} $\triv_G$ equal to $\C$ as a vector space with every $g \in G$ acting as $1$. This corresponds to the {\em trivial representation}, namely, the homomorphism mapping every $g \in G$ to $1 \in GL_1(\C)$. Let us always choose $L_1$ in (\ref{EIrrG}) to be the trivial module. 
\end{Example}

\begin{Example}\label{ExS_3DimRep}
Let $G = S_3$. There are three conjugacy classes.
Hence $r = 3$, i.e. there are three isomorphism classes of irreducible
$\C S_3$-modules.
Moreover $n_1^2 + n_2^2+n_3^2 = 6$ so the dimensions
of irreducible modules can only be $1,1$ and $2$.
\end{Example}

\begin{Example}
There is a group homomorphism
$\sgn: S_n \rightarrow \{\pm 1\}\subseteq \C^\times$. One can view this as a 
$1$-dimensional representation, the {\em sign representation}.
The corresponding module is not isomorphic to the trivial module (providing $n > 1$). Recall Example~\ref{ExS_3DimRep}. Now we have constructed both of the $1$-dimensional $\C S_3$-modules:
one is trivial, the other is sign. What about the $2$-dimensional irreducible $\C S_3$-module?
\end{Example}

\begin{Example} %\label{}%{\rm \cite{}}%{\bf ()}
{\rm 
Let $X = \{x_1,\dots,x_n\}$ be a finite $G$-set and $\C X$ be the corresponding permutation $\C G$-module. This defines a representation $\rho:G\to GL(\C X)$. 
%Thus, each $g \in G$ is viewed as an element $\rho(g)$ of $GL(\C X)$. 
Note the $ij$-entry of the matrix of $\rho(g)$ in the natural basis of $\C X$ is $1$ if $g x_j = x_i$ and it is zero otherwise. This means that the matrix $\rho(g)$ is a
{\em permutation matrix}: all its entries are zeros and ones, and there is just one non-zero entry in every row and column. So amongst all matrix representations of $G$, the ones coming from permutation representations are in a sense very easy... On the other hand,  $\C X$ %permutation module 
is {\em not}\, irreducible unless $n=1$ (why?).  
}
\end{Example}

If $V$ and $W$ are two finite dimensional
$\C G$-modules then so is $V \oplus W$. Pick bases of $V$ and $W$, to view $V$ as
a matrix representation $\rho: G \rightarrow GL_n(\C)$ and $W$ as a matrix
representation $\sigma: G \rightarrow GL_m(\C)$. With respect to the
basis for $V \oplus W$ obtained by concatenating the two bases, 
the matrix representation $\rho \oplus \sigma: G \rightarrow GL_{m+n}(\C)$
corresponding to the module $V \oplus W$ has all $g$ mapping to block 
diagonal matrices $\diag(\rho(g), \sigma(g))$. This is how one could think of direct sums of $\C G$-modules in terms of matrices.

\begin{Example}\label{ExPermS3}
Let us go back to $S_3$ again. It acts on $X = \{1,2,3\}$ and so has a permutation representation $\C X$. %The corresponding matrix representation $S_3 \rightarrow GL_3(\C)$ is one you know perfectly well. 
For instance, the image of 
the $3$-cycle $(1\, 2\, 3)$ with respect to the standard basis $v_1,v_2,v_3$
of $\C X$
labelled by the elements of the set $X$ 
is the matrix
$$
\left(\begin{array}{lll}
0 & 0 & 1\\
1& 0 & 0\\
0&1&0
\end{array}\right).
$$
Note that the vector $v_1+v_2+v_3$ is fixed by %all of 
$G$, so it spans a 
$1$-dimensional submodule, isomorphic to the trivial module. 
By Maschke's Theorem that had better have a complement. For instance, the set of all vectors 
$a_1v_1+a_2v_2+a_3v_3$ with $a_1+a_2+a_3 = 0$, which is $\spa(v_1-v_2, v_2-v_3)$, is a complement. Let us write down matrices
with respect to the new basis $v_1+v_2+v_3, v_1-v_2, v_2-v_3$ instead: 
\begin{align*}
\rho(1) &= \left(\begin{array}{lll}
1 & 0 & 0\\
0& 1 & 0\\
0&0&1
\end{array}\right),
\quad 
\rho((1\,2)) = \left(\begin{array}{lll}
1 & 0 & 0\\
0& -1 & 1\\
0&0&1
\end{array}\right),\\
\rho((2\,3)) &= \left(\begin{array}{lll}
1 & 0 & 0\\
0& 1 & 0\\
0&1&-1
\end{array}\right),
\quad
\rho((1\,3)) = \left(\begin{array}{lll}
1 & 0 & 0\\
0& 0 & -1\\
0&-1&0
\end{array}\right),\\
\rho((1\,2\,3)) &= \left(\begin{array}{lll}
1 & 0 & 0\\
0& 0 & -1\\
0&1&-1
\end{array}\right),\quad
\rho((1\,3\,2)) = \left(\begin{array}{lll}
1 & 0 & 0\\
0& -1 & 1\\
0&-1&0
\end{array}\right).
\end{align*}
Note all these matrices are block diagonal. The top $1 \times 1$ block is the trivial representation of $G$ on $V_1 = \C (v_1+v_2+v_3)$, 
the bottom $2 \times 2$ block is a $2$-dimensional representation of $G$ on $V_2 = \spa(v_1-v_2,v_2-v_3)$. 
It is easy to check that $V_2$ is irreducible.
The decomposition $V = V_1 \oplus V_2$ of $V$ into irreducibles corresponds in matrix language to choosing a basis so that each $\rho(g)$ is block diagonal, and 
since the blocks are irreducible representations you cannot do any better. Note, by the way, that we have found the ``missing'' irreducible $\C S_3$-module of dimension $2$. 
\end{Example}

%We classify irreducible complex representations of the dihedral group $D_{2n}=\lan a,b\mid a^n=b^2=1,\ bab^{-1}=a^{-1}\ran$. 

Recall from Exercise~\ref{ExHopfAlg} that $\C G$ is a co-commutative Hopf algebra and so there is a natural structure of $\C G$-module on the tensor product of two $\C G$-modules as well as on a dual of a $\C G$-module. To be more precise, let $V$ and $W$ be $\C G$-modules. Then $V\otimes W$ (which means $V\otimes_\C W$) is a $\C G$-module with the action $g(v\otimes w)=gv\otimes gw$ for all $g\in G$ and $v\in V, w\in W$. Also, $V^*$ is a $\C G$-module with the action $gf(v)=f(g^{-1}v)$ for all $g\in G$, $f\in V^*$, and $v\in V$. These operations satisfy all the natural properties discussed in Exercise~\ref{ExHopfAlg}.

The tensor product discussed above should not be confused with the {\em outer tensor product} which arises as follows. Given {\em two} groups $G$ and $H$, a $\C G$-module $V$ and $\C H$-module $W$, their outer tensor product $V\boxtimes W$ is the vector space $V\otimes W$ considered as a $\C[G\times H]$-module via $(g,h)(v\otimes w)=gv\otimes gw$ for all $g\in G, h\in H, v\in V, w\in W$. In view of Exercise~\ref{ExGrAlDirProd}, this is a special case of the outer tensor product construction for  associative algebras studied in Exercises~\ref{ExOuterTensProdAlg}, \ref{LHomOuterTens}, and~\ref{TOuterTensAlg}. 

Powerful tools to build new representations from old ones are provided by restriction and induction. Let $H$ be a subgroup of a finite group $G$. Denote the category of {\em finite dimensional} $\C G$-modules (and usual $\C G$-module homomorphisms) by $\mod{\C G}$. We have the {\em restriction} and {\em induction} functors 
$$
\res^G_H:\mod{\C G}\to\mod{\C H}, \quad \ind^G_H:\mod{\C H}\to\mod{\C G}. 
$$ 
As a special case of a general fact, $\ind^G_H$ is left adjoint to $\res^G_H$.

Let $g_1,\dots,g_m$ be the left coset representatives of $H$ in $G$. Then $\C G$ is a free right $\C H$-module with basis $\{g_1,\dots,g_m\}$. By Exercise~\ref{ExCanTens}, we have a vector space decomposition:
$$
\ind_H^G V = g_1\otimes V\oplus\dots\oplus g_m\otimes V.
$$
So:

\begin{Lemma} \label{LIndProp}%{\rm \cite{}}%{\bf ()}
Let $G$ be a finite group, $H\leq G$, and $V$ be a finite dimensional $\C H$-module. 
\begin{enumerate}
\item[{\rm (i)}] $\dim\ind_H^G V=[G: H]\dim V$. 
\item[{\rm (ii)}] If $\{v_1,\dots,v_n\}$ is a basis of $V$ then $\{g_i\otimes v_j\mid 1\leq i\leq m,1\leq j\leq n\}$ is a basis of $\ind_H^G V$. 
\end{enumerate}
\end{Lemma}

A representation $\rho:G\to GL_n(\C)$ associates to every element $g\in G$ an $n\times n$ matrix. The set of matrices $\{\rho(g)\mid g\in G\}$ is a ``lot of data'' to carry. Miraculously, it turns out that a representation is determined uniquely up to isomorphism by its character: %, which is a $\C$-valued function on $G$, constant on conjugacy classes.

\begin{Definition}%\label{}%{\rm \cite{}}%{\bf ()}
{\rm 
Let $V$ be a finite dimensional $\C G$-module and $(V, \rho)$ be the corresponding  representation of $G$. The {\em character of $V$} is the function 
$\chi_V: G \rightarrow \C$ with $\chi_V(g)$ equal to the trace of the endomorphism
$\rho(g)$. 
}
\end{Definition}

%Whenever convenient, we extend a character to a function on $\C G$ by linearity. 
It is clear that if $V\cong W$ then $\chi_V=\chi_W$. The miracle is that the converse is also true! Clearly $\chi_V(1)=\dim V$. It is also easy to check that $\chi_{V\oplus W}=\chi_V+\chi_W$ and $\chi_{V\otimes W}=\chi_V\chi_W$.

A {\em class function} on $G$ is a function $f:G\to \C$ that is {\em constant on conjugacy classes}.
For example, the character $\chi_V$ of any finite dimensional $\C G$-module
is a class function, since 
$$
tr(\rho(hgh^{-1})) = tr(\rho(h) \rho(g) \rho(h)^{-1})
%= tr (\rho(h)^{-1} \rho(h) \rho(g)) 
= tr (\rho(g))\qquad(g,h\in G).
$$
%hence $\chi_V(hgh^{-1}) = \chi_V(g)$.

Let $C(G)$ denote the vector space of all class functions on $G$, 
%There is an obvious basis for $C(G)$: l
and let 
\begin{equation}\label{EConCl}
C_1,\dots,C_r
\end{equation}
be the conjugacy classes of $G$. 
We will always choose $C_1$ to be the trivial conjugacy class: 
$C_1 = \{1\}$.
Let $\delta_i:G \rightarrow \C$ be the function with
$\delta_i(g) = 1$ if $g \in C_i$, $0$ otherwise.
Then $\{\delta_1,\dots,\delta_r\}$ is clearly a basis of $C(G)$.
There is a much less obvious (and often more important!) basis for $C(G)$, coming from irreducible characters. 

%Recall from (\ref{EIrrG}) that we have picked representatives $L_1,L_2,\dots, L_r$ of the isomorphism classes of irreducible $\C G$-modules with $L_1\cong\triv_G$. 

The characters of the irreducible modules $L_1,\dots,L_r$ are called the {\em irreducible characters of $G$}. They will be denoted, respectively, by  
\begin{equation}\label{EIrrChar}
\chi_1,\dots,\chi_r.
\end{equation}

%Recall the central idempotents $e_1,\dots,e_r$ from (\ref{EIdG}). 

\begin{Theorem} 
$\chi_1,\dots,\chi_r$ is a basis for $C(G)$.
\end{Theorem}
\begin{proof}
%Consider the Wedderburn-Artin decomposition (\ref{wd}). 
%In the product on the right hand side, let $e_i := (0,\cdots,0, I_{n_i}, 0, \cdots, 0)$ be the identity matrix of the $i$th matrix algebra.
%Recall that $L_i$ is the module of column vectors for $M_{n_i}(\C)$. 
% on which the idempotent $e_i$ acts as $1$ and all other $e_j$ act as zero.
%So $\chi_i(x)$ is just the trace of the $i$th matrix in the Wedderburn decomposition of an element of $x \in \C G$. 
%But 
By definition, $\chi_i(e_j) = \delta_{i,j} n_j$.
This proves that $\chi_1,\dots,\chi_r$ are linearly independent.
Hence they form a basis by dimensions.
\end{proof}

\begin{Corollary} Two finite dimensional $\C G$-modules $V$ and $W$
are isomorphic if and only if $\chi_V = \chi_W$, i.e. they have the same
characters.
\end{Corollary}
\begin{proof}
By Maschke's Theorem, %$V$ is a direct sum of irreducibles, say 
$V \cong \bigoplus_{i=1}^r L_i^{\oplus a_i}$ and 
%Similarly
$W \cong \bigoplus_{i=1}^r L_i^{\oplus b_i}$.
By the Jordan-H\"older Theorem, $V \cong W$ if and only if $a_i = b_i$ for all $i$. But $\chi_V = \sum_{i=1}^r a_i \chi_i$
and $\chi_W = \sum_{i=1}^r b_i \chi_i$.
So $\chi_V = \chi_W$ if and only if $a_i = b_i$ for all $i$ by the linear independence of irreducible characters. 
\end{proof}

\iffalse{
In the proof of the theorem, we exploited the idempotents
$e_i \in \C G$ coming from the identity matrices in the Wedderburn decomposition. Let us repeat this, because these are important: there are 
mutually orthogonal central idempotents 
\begin{equation}\label{EIdG}
e_1,\dots,e_r \in \C G
\end{equation}
summing to the identity
with the property that $e_i$ acts on the $j$th irreducible module $L_j$ 
as $\delta_{i,j}$. 
}\fi

%Recall the central idempotents (\ref{EIdG}). There are explicit formulas for them in terms of irreducible characters:

\begin{Lemma}\label{ECIForm}
For each $i=1,\dots,r$ we have  $e_i = \sum_{g \in G} \frac{n_i \chi_i(g^{-1})}{|G|} g$.
% then $$a_g^{(i)} = \frac{n_i \chi_i(g^{-1})}{|G|}.$$
\end{Lemma}
\begin{proof}
Let us write $e_i = \sum_{g \in G} a_g^{(i)} g$, let $\psi$ be the character of the regular $\C G$-module and $g \in G$.
We compute $\psi(e_i g^{-1})$ in two different ways.
On the one hand,
$$
e_i g^{-1} = \sum_h a_h^{(i)} hg^{-1}.
$$
So, since $\psi(1) = |G|$ and $\psi$ is zero on all other group elements, we have 
$$
\psi(e_i g^{-1}) = a_g^{(i)} |G|.
$$
On the other hand, $\psi = \sum_{j=1}^r n_j \chi_j$.
So 
$$
\psi(e_i g^{-1}) = \sum_j n_j \chi_j(e_i g^{-1}).
$$
But $e_i$ acts as zero on all $L_j$ for $j \neq i$, and it acts
as $1$ on $L_i$. So we get that
$$
\psi(e_i g^{-1}) = n_i \chi_i(g^{-1}).
$$
Comparing the two formulas proves the lemma.
\end{proof}

Recall that in (\ref{EConCl}) we have denoted the conjugacy classes of $G$ by $C_1,\dots, C_r$. Let $c_i = |C_i|$ be the size of the $i$th conjugacy class, so e.g. $c_1=1$. Also, pick once and for all a representative $g_i$ in each conjugacy class
$C_i$.

\begin{Definition}%\label{}%{\rm \cite{}}%{\bf ()}
{\rm 
The {\em character table} of $G$
is the $r \times r$ matrix with the $(i,j)$-entry equal to
$\chi_i(g_j)$. 
}
\end{Definition}

It is convenient to think of the rows of the character table of $G$ as being labeled by the irreducible characters of $G$ and the columns being labelled by the conjugacy classes of $G$. The character table is independent of the particular representative
$g_j$ of $C_j$ chosen because $\chi_i$ is a class function.

We now introduce a Hermitian form on the complex vector space $C(G)$ by
defining the pairing of two class functions $\chi$ and $\psi$ as follows:
\begin{equation}\label{EHermFormChar}
(\chi, \psi) = \frac{1}{|G|} \sum_{g \in G} \chi(g) \overline{\psi(g)}\qquad(\chi,\psi\in  C(G)),
\end{equation}
where ``$\,\bar{\ }\,$'' is the  complex conjugation.
Note that 
$(\chi,\chi) = \frac{1}{|G|} \sum_{g \in G} |\chi(g)|^2$ which is a positive real number 
if and only if $\chi \neq 0$. 
So we have a positive definite Hermitian form or an {\em inner product}.

\begin{Theorem} \label{TOrtRel}%{\rm \cite{}}%
{\bf (Character Orthogonality Relations)}
\begin{enumerate}
\item[{\rm (i)}] With respect to the inner product just defined, 
$\chi_1,\dots,\chi_r$ are orthonormal. In particular for any character $\chi$, we have 
$$
\chi = \sum_{i=1}^r (\chi, \chi_i) \chi_i.
$$
\item[{\rm (ii)}] {\bf (Row Orthogonality Relations)} With our usual notation for the character table, we have for any $1\leq i,j\leq r$ that
$$
\sum_{k=1}^r c_k \chi_i(g_k) \overline{\chi_j(g_k)}
= \left\{
\begin{array}{ll}
0 & \text{if $i \neq j$},\\
|G|&\text{if $i = j$.}
\end{array}
\right.
$$
\item[{\rm (iii)}] {\bf (Column Orthogonality Relations)} For any $1\leq i,j\leq r$ we have 
$$
\sum_{i=1}^r \chi_i(g_j) \overline{\chi_i(g_k)}
= \left\{
\begin{array}{ll}
0 & \text{if $j \neq k$},\\
|G| / c_j&\text{if $j = k$.}
\end{array}
\right.
$$
\end{enumerate}
\end{Theorem}
\begin{proof}
(i) Clearly $\chi_i (e_j) = \delta_{i,j} n_j$. On the other hand, by Lemma~\ref{ECIForm}, we have  $e_j = \frac{1}{|G|} \sum_g n_j \chi_j(g^{-1}) g$.
Hence
$$
\delta_{i,j} n_j = \frac{1}{|G|} \sum_g n_j \chi_j(g^{-1}) \chi_i(g).
$$
It is easy to see that $\chi_j(g^{-1})=\overline{\chi_j(g)}$. 
Hence the right hand side is $n_j (\chi_i, \chi_j)$.

(ii) This is just a restatement of (i) using the fact that characters are class functions. 

(iii)
Let $A$ be the character table, 
%i.e. the matrix with $ij$-entry $\chi_i(g_j)$. Let 
and $B$ be the matrix with $ij$-entry
$c_i \overline{\chi_j(g_i)} / |G|$.
%Let us compute $AB$: its
The $ij$-entry of the matrix $AB$ is
$$
\frac{1}{|G|}
\sum_{k=1}^r 
c_k
\chi_i(g_k)
\overline{\chi_j(g_k)} 
= \delta_{i,j}.
$$
So $AB = I$. Hence $BA = I$. Now, computing the $ij$-entry of $BA$, we get
$$
\frac{1}{|G|}
\sum_{k} c_i \overline{\chi_k(g_i)} \chi_k(g_j)
= \delta_{i,j}.
$$
We are done. 
\end{proof}

Finally, we mention one important corollary of character theory without proof: the dimension of any irreducible $\C G$ module divides the order of the group $G$.

\section{Gelfand-Zetlin subalgebras and Gelfand-Zetlin bases}\label{S2}
We now begin to study representation theory of the symmetric groups $S_n$ in more detail. Our first approach will be the one suggested by Okounkov and Vershik, but it relies on many classical ideas going back at least to Young. We are going to exploit the following three vague general ideas:

\begin{enumerate}
\item[$\bullet$] We want to understand the Wedderburn decomposition more explicitly in terms of the data of the symmetric group. 
\item[$\bullet$] There is more than one symmetric group, in fact we have a nice nested family $S_1\subset S_2\subset S_3\subset\dots$. 
\item[$\bullet$] The symmetric group algebra $\C S_n$ has a well-known large commutative subalgebra, using which we can try to ``play Lie theory" as if this subalgebra was a maximal toral subalgebra. 
\end{enumerate}

We now expand on the last point. 

Define the $k$th {\em Jucys-Murphy element} (JM-element for short) $L_k\in \C S_n$ as follows:
\begin{equation}\label{EJM}
L_k:=\sum_{1\leq m<k}(m,k).
\end{equation}
Note that $L_1=0$ and $L_k$ commutes with $S_{k-1}$. 
As $L_k\in \C S_k$, it follows that the JM-elements commute.
Here and below, if $m<n$, the default embedding of $S_{m}$ into $S_n$ is with respect to the {\em first} $m$ letters. A copy of  $S_{m}$ embedded with respect to the {\em last} $m$ letters is denoted by $S_m'$.

Denote by $Z_n$ the center of the group algebra $\C S_n$. Also let
$$
Z_{n,m}:=(\C S_{n+m})^{S_n}
$$
be the centralizer of $\C S_n$ in $\C S_{n+m}$. 
It is clear that $Z_{n,m}$ has basis consisting of the class sums corresponding to the $S_n$-conjugacy classes in $S_{n+m}$. These conjugacy classes can be thought of as cycle shapes with `fixed positions' for $n+1,n+2,\dots,n+m$---we call them {\em marked cycle shapes}. For example, the symbol
\begin{equation}\label{E270603}
(*,*,*,*,*)(*,*)(*)(*)(12,*,13,14,*)(15)
\end{equation}
corresponds to the $S_{11}$-conjugacy class in $S_{15}$ which consists of all permutations whose cycle presentation is obtained by inserting the numbers $1$ through $11$ instead of asterisks. 
We denote by 
\begin{equation*}\label{E270603*}
[(*,*,*,*,*)(*,*)(*)(*)(12,*,13,14,*)(15)]\in Z_{11,4}
\end{equation*}
the corresponding class sum. 

Let $Z_{n,m}[i]$ denote the span of the class sums which consists of permutations fixing exactly $n+m-i$ elements (equivalently, moving exactly $i$ elements), and 
$$Z_{n,m}^i:=Z_{n,m}[0]+Z_{n,m}[1]+\dots+Z_{n,m}[i].$$
Then we have a vector space decomposition 
$$
Z_{n,m}=\bigoplus_{i\geq 0} Z_{n,m}[i],
$$
and the algebra filtration 
$$
\C\cdot 1=Z_{n,m}^0=Z_{n,m}^1\subseteq Z_{n,m}^2\subseteq\dots.
$$
\iffalse{
Note that for $k>n$ and $r\leq n$, we have that 
$$
[(*,k)]^r=[(\underbrace{*,\dots,*}_{r},k)]+\text{lower terms}. 
$$
Hence for two distinct elements $k_1,k_2>n$, and $r\leq n$, we have
$$
(k_1,k_2)[(*,k_2)]^r=[(k_2,\underbrace{*,\dots,*}_{r},k_1)]+\text{lower terms}. 
$$
Continuing like this, we see that for any distinct elements $k_1,\dots,k_t>n$, and any non-negative $r_1,\dots,r_t$ with $r_1+\dots+r_t\leq n$, we can get a class sum corresponding to any marked cycle modulo lower terms as follows: 
\begin{equation}\label{E250513}
\begin{split}
&[(*,k_t)]^{r_t}(k_t,k_{t-1})\dots [(*,k_2)]^{r_2}(k_2,k_1)[(*,k_1)]^{r_1}
\\
&=[(\underbrace{*,\dots,*}_{r_t},k_t,\underbrace{*,\dots,*}_{r_{t-1}},\dots,k_{2},\underbrace{*,\dots,*}_{r_1},k_1)]+\text{lower terms}. 
\end{split}
\end{equation}
This in turn implies 
}\fi

\begin{Lemma} \label{LNew}%{\rm \cite{}}%{\bf ()}
Let $x$ be a marked cycle shape which corresponds to the $S_n$-conjugacy class in $S_{n+m}$, consisting of permutations moving exactly $i$ elements, i.e. $x\in Z_{n,m}^{i}\setminus Z_{n,m}^{i-1}$. Then, modulo $Z_{n,m}^{i-1}$, the class sum $[x]$ can be written as a product of an element of $Z_n\subseteq  Z_{n,m}$ and elements of the form $[(*,k)]$ and $(k',k'')$ for some $k,k',k''>n$. 
\end{Lemma}
\begin{proof}
First of all, note that the problem reduces to the case where $x$ is  just one cycle, involving some $k>n$, i.e. where $x$ is of the form
$$
x=(\underbrace{*,\dots,*}_{r_t},k_t,\underbrace{*,\dots,*}_{r_{t-1}},\dots,k_{2},\underbrace{*,\dots,*}_{r_1},k_1).
$$
Now, observe that 
$$
[x]\equiv 
[(*,k_t)]^{r_t}(k_t,k_{t-1})\dots [(*,k_2)]^{r_2}(k_2,k_1)[(*,k_1)]^{r_1}
\pmod{Z_{n,m}^{i-1}},
$$
using the fact that $Z_{n,m}$ is an algebra. 
\end{proof}

\iffalse{
To be more precise, a {\em marked permutation} of the set $\{n+1,n+2,\dots,n+m\}$ is a permutation of this set written in arrow notation with arrows labelled  with non-negative integers. For example,
$$
(n+1\stackrel{1}{\rightarrow}n+2\stackrel{0}{\rightarrow}n+1)(n+3\stackrel{2}{\rightarrow}n+3)
$$
is a marked permutation of $\{n+1,n+2,n+3\}$. If the sum of all arrow labels in a marked permutation is equal to $n$, we associate to it an $S_n$-conjugacy class in $S_{n+m}$. For example the marked permutation above corresponds to the $S_3$-conjugacy class in $S_{6}$ which consists of all permutations whose cycle shape is
$$
(n+1
$$
}\fi

\begin{Proposition}\label{POl} %{\rm \cite{OlY}}%
{\bf (Olshanskii's Lemma)}
The algebra $Z_{n,m}$ is generated by $S_m'$, $Z_n$, and $L_{n+1},\dots,L_{n+m}$. 
\end{Proposition}
\begin{proof}
It is clear that $S_m'$, $Z_n$, and $L_{n+1},\dots,L_{n+m}$  are contained in $Z_{n,m}$, so they generate a subalgebra $A\subseteq Z_{n,m}$. Conversely, we prove by induction on $i=0,1,\dots$ that $Z_{n,m}^i\subseteq A$. For $i=0$ and $1$, we have $Z_{n,m}^i=F\cdot 1\subseteq A$. Note that for any $k>n$, we have
$$
[(*,k)]=L_k-(n+1,k)-\dots-(k-1,k)\in A. 
$$
So it follows from Lemma~\ref{LNew} that, modulo $Z_{n,m}^{i-1}$, we can write any class sum $[x]\in Z_{n,m}^i$ as a product of elements in $A$. But $Z_{n,m}^{i-1}\subseteq A$ by the inductive assumption, and we are done. 
%We explain the inductive step on example. Let $z\in Z_{11,4}^{12}$ be the class sum corresponding to  the marked cycle shape from (\ref{E270603}). Let $c\in Z_{11}$ denote the sum of all elements of $S_{11}$ whose cycle shape is 
%$$
%(*,*,*,*,*)(*,*)(*)(*).
%$$
%Also, let 
%$$
%x=(12,13)L_{12}(13,14)(L_{14}-(12,14)-(13,14))\in A.
%$$
%(Note that $L_{12}$ is the class sum corresponding to $(*,12)$, and $(L_{14}-(12,14)-(13,14))$ is the class sum corresponding to  $(*,14)$). 
%Then $xc$ is equal to $z$ modulo lower layers of our filtration. 
\end{proof}

Let $B$ be a subalgebra of an $F$-algebra $A$ and $C$ be the centralizer of $B$ in $A$. If $V$ is an $A$-module and $W$ is a $B$-module then $\Hom_B(W,\res_B V)$ is naturally a  $C$-module with respect to the action 
$(cf)(w)=cf(w)$ for $w\in W, f\in  \Hom_B(W,\res_B V), c\in C$. 

\begin{Lemma}\label{L270603}%{\rm \cite{}}%{\bf ()}
Let $B\subseteq A$ be semisimple finite dimensional $F$-algebras. If $V$ is irreducible over $A$ and $W$ is irreducible over $B$ then $$\Hom_B(W,\res_B V)$$ is irreducible over $C$. 
\end{Lemma}
\begin{proof}
By Wedderburn-Artin, we may assume that $A=\End(V)$.  Decompose $\res_B V=W^{\oplus k}\oplus X$, where $W$ is not a composition factor of $X$. Then the algebra $\End_B(W^{\oplus k})$, naturally contained in $C$, acts on the space $\Hom_B(W,\res_B V)$ as the full endomorphism algebra.  
\end{proof}

\begin{Theorem}\label{TBrMF}%{\rm \cite{}}%{\bf ()}
Let $V$ be an irreducible $\mathbb{C} S_n$-module. Then the restriction $\res_{S_{n-1}} V$ is multiplicity-free.
\end{Theorem}
\begin{proof}
It follows from Proposition~\ref{POl} that the centralizer of $\mathbb{C}S_{n-1}$ in $\mathbb{C}S_n$ is commutative. So the theorem comes from Lemma~\ref{L270603} (why?). 
\end{proof}

We now define the {\em branching graph} $\mathbb{B}$ whose vertices are isomorphism classes of irreducible $\mathbb{C}S_n$-modules for all $n\geq 0$ (by convention $\mathbb{C}S_0=\mathbb{C}$); we have a directed edge $W\to V$ from (an isoclass of) an irreducible $\mathbb{C}S_n$-module $W$ to (an isoclass of) an irreducible $\mathbb{C}S_{n+1}$-module $V$ if and only if $W$ appears as a composition factor of  $\res_{S_{n}} V$; there are no other edges. 

Our main goal is to find an explicit combinatorial description of the branching graph. This will give us a good understanding of 
%labeling of the 
irreducible $\mathbb{C} S_n$-modules {\em for all $n$}. 
This will also yield the so-called {\em branching rule}, i.e. the rule that describes a restriction of an irreducible complex $S_n$-representation to $S_{n-1}$. To achieve this goal we will actually do more. 

Let $V$ be an irreducible $\mathbb{C}S_n$-module. 
Theorem~\ref{TBrMF} and Exercise~\ref{ExCan} imply that the decomposition 
$$
\res_{S_{n-1}} V=\bigoplus_{W\to V} W
$$
is canonical. Decomposing each $W$ on restriction to $S_{n-2}$, and continuing inductively all the way to $S_0$, we get a canonical decomposition
$$
V=\bigoplus_{T} V_T
$$
into irreducible $\mathbb{C}S_0$-modules, that is $1$-dimensional subspaces $V_T$, where $T$ runs over all paths $W_0\to W_1\to\dots\to W_n=V$ in $\mathbb{B}$. 

%(For technical convenience we have introduced here the symmetric group $S_0$, which, like $S_1$, is the trivial group.) 

Let $T=(W_0\to W_1\to\dots\to W_n)$. Note that for all $0\leq k\leq n$ we have 
\begin{equation}\label{E010703}
\C S_k\cdot V_T=W_k, \quad\text{and} \quad 
W_T=e_{W_0}e_{W_{1}}\dots e_{W_n}V.
%(T=(W_0\to W_1\to\dots\to W_n),\ 0\leq k\leq n).
\end{equation}

Choosing a vector $v_T\in V_T$, %such that $(v_T,v_T)=1$, 
we get a basis $\{v_T\}$ of $V$ called {\em Gelfand-Zetlin basis} (or GZ-basis). 
%If  $(\cdot,\cdot)$ is fixed, 
Vectors of GZ-basis are defined uniquely up to scalars.  
%Otherwise, there is an additional freedom of multiplying all basis vecors by the same scalar. 
Moreover, if $\phi:V\to V'$ is an isomorphism of irreducible modules then $\phi$ moves a GZ-basis of $V$ to a GZ-basis of $V'$. 

Using Exercise~\ref{ExScPr}, pick an $S_n$-invariant inner product $(\cdot,\cdot)$ on $V$. %; it is unique up to a scalar. 
Note, for example using (\ref{E010703}),  that a $GZ$-basis is orthogonal with respect to $(\cdot,\cdot)$. 

%Now decompose the algebra $\mathbb{C} S_n$ according to the Wedderburn- Artin Theorem
%\begin{equation}\label{EWed}
%\mathbb{C} S_n=\bigoplus_V \End_\mathbb{C}(V),
%\end{equation}
%where the sum is over the representatives of the isoclasses of irreducible $\mathbb{C} S_n$-modules. This decomposition is canonical. Let us pick a $GZ$-basis in each $V$. 
Choice of a $GZ$-basis in each irreducible module $V$ yields by Weddreburn-Artin, a decomposition 
\begin{equation}\label{EWedM}
\mathbb{C} S_n=\bigoplus_V M_{\dim V}(\mathbb{C}). 
\end{equation}

Define the {\em GZ-subalgebra} $A_n\subseteq\C S_n$ as the subalgebra which consists of all elements of $\mathbb{C} S_n$ which are diagonal  with respect to a $GZ$-basis in every irreducible $\mathbb{C} S_n$-module. In terms of the decomposition (\ref{EWedM}), $A_n$ consists of all diagonal matrices. In particular,

\begin{Lemma}%\label{}%{\rm \cite{}}%{\bf ()}
$A_n$ is a maximal commutative subalgebra of $\mathbb{C} S_n$. Also, $A_n$ is a semisimple algebra. 
\end{Lemma}

%Note that we have only defined GZ-bases and GZ-subalgebras over $\mathbb{C}$. `Analogous' subalgebra in characteristic $p$ will not be semisimple. 
We now give two more explicit descriptions of the GZ-subalgebra.

\begin{Lemma}%\label{}%{\rm \cite{}}%{\bf ()}
We have
\begin{enumerate} 
\item[{\rm (i)}] $A_n$ is generated by the subalgebras $Z_0,Z_1,\dots,Z_n\subseteq \C S_n$. 
\item[{\rm (ii)}] $A_n$ is generated by the JM-elements $L_1,L_2,\dots,L_n$.
\end{enumerate}
\end{Lemma}
\begin{proof}
(i)
Let $e_V\in Z_n$ be the central idempotent of $\mathbb{C} S_n$ which acts as identity on $V$ and as zero on any irreducible $\mathbb{C} S_n$-module $V'\not\cong V$. If $T=W_0\to W_1\to\dots\to W_n=V$ is a path in $\mathbb{B}$ then 
$$
e_{W_0}e_{W_1}\dots e_{W_n}\in Z_0Z_1\dots Z_n
$$ 
acts as the  projection to $V_T$ along $\oplus_{S\neq T}V_S$ and as zero on any irreducible $\mathbb{C} S_n$-module $V'\not\cong V$. So the subalgebra generated by $Z_0,Z_1,\dots,Z_n$ contains $A_n$. As this subalgebra is commutative and $A_n$ is a maximal commutative subalgebra of $\mathbb{C} S_n$, the two must coincide. 

(ii) Note that  $L_k$ is the sum of all transpositions in $S_k$ minus the sum of all transpositions in $S_{k-1}$, that is $L_k$ is a difference of a central element in $S_k$ and a central element in $S_{k-1}$. So by (i), the JM-elements do belong to $A_n$. To prove that they generate $A_n$, proceed by induction on $n$, the inductive base being trivial. By (i), $A_n$ is generated by $A_{n-1}$ and $Z_n$. In view of the inductive assumption, it suffices to prove that $A_{n-1}$ and $L_n$ generate $Z_n$. But this follows from the obvious embedding $Z_n\subseteq Z_{n-1,1}$ and Proposition~\ref{POl}, as $Z_{n-1}\subseteq A_{n-1}$.
\end{proof}

Now, we will try to have the GZ-subalgebra play a role of a Cartan subalgebra in Lie theory. As $A_{n}$ is semisimple we can decompose every irreducible $\mathbb{C} S_n$-module $V$ as a direct sum of simultaneous eigenspaces for the elements $L_1,\dots,L_n$. If $\bi=(i_1,\dots,i_n)\in\mathbb{C}^n$ and $V_\bi$ is the simultaneous eigenspace for the $L_1,\dots,L_n$ corresponding to the eigenvalues $i_1,\dots,i_n$, respectively, then we say that $\bi$ is a {\em weight} of $V$ and $V_\bi$ is the {\em $\bi$-weight space} of $V$. 

By definition, vectors of a $GZ$-basis are weight vectors. Also, since in terms of (\ref{EWedM}), $A_n$  consists of {\em all} diagonal matrices, each weight space is $1$-dimensional. Thus the weight spaces are precisely the spans of the elements of a GZ-basis. It also follows that if $\bi$ is a  weight of an irreducible module $V$, then it is {\em not} a weight of an irreducible module $V'\not\cong V$.  
Thus, via GZ-bases, we get a one-to-one correspondence between all possible weights (for all symmetric groups) and all paths in $\mathbb{B}$. The weight corresponding to a path $T$ will be denoted $\bi^T$ and a path corresponding to a weight $\bi$ will be denoted $T_\bi$. We will also write $v_\bi$ for $v_{T_\bi}$. 

A path $T$ ends at a vertex $V$ if and only if the corresponding weight $\bi^T$ is a weight of $V$. It is clear now that in order to understand $\mathbb{B}$ it suffices to describe the sets
\begin{equation}\label{EWei}
W(n)=\{\bi\in\mathbb{C}^n\mid \bi\ \text{is a weight of a $\mathbb{C} S_n$-module}\}\quad (n\geq 0)
\end{equation}
and the equivalence relation 
\begin{equation}\label{EEqu}
\bi\approx\bj \ \equivalent\  \text{$\bi$, $\bj$ are weights of the same irreducible $\mathbb{C} S_n$-module}
\end{equation}
on $W(n)$. Indeed, note that 
$$
\mathbb{B}=\bigsqcup_{n\geq 0} (W(n)/\approx)
$$
and for equivalence classes $[\bi]\in W(n-1)/\approx$, $[\bj]\in W(n)/\approx$, we have $[\bi]\to[\bj]$ if and only if $\bi=(k_1,\dots,k_{n-1})$ for some $\bk\approx\bj$. 

\begin{Remark}%\label{}%{\rm \cite{}}%{\bf ()}
{\rm 
For those of you who are spoiled by knowing what the final answer should be: yes, the elements of the set $W(n)/\approx$ will be labeled by the partitions $\la$ of $n$, and the elements of the set $W(n)$ will be labeled by the standard $\la$-tableaux for all such $\la$, with two tableaux being equivalent if and only if they have the same shape $\la$. To be more precise, if $T$ is a standard $\la$-tableaux, then the corresponding weight $\bi^T=(i_1,\dots,i_n)$ is obtained as follows: $i_r$ is the content of the box in $\la$ which is occupied by $r$ in the $\la$-tableaux $T$ ($1\leq r\leq n$).
}
\end{Remark}

The following notation will be convenient: if $\bi=(i_1,\dots,i_n)\in W(n)$, we write $V(\bi)$ for an irreducible $\C S_n$-module which has $\bi$ as its weight. The weight $\bi$ determines $V(\bi)$ uniquely up to isomorphism, but
$
V(\bi)\cong V(\bj)\quad\text{if and only if $\bi\approx \bj$}.
$
Now, (\ref{E010703}) can now be restated as follows:
\begin{equation}\label{E010703_2}
\C S_k\cdot v_\bi=V(i_1,\dots,i_k) \qquad(0\leq k\leq n).
\end{equation}

\chapter{Day Two}

\section{Description of weights}

We have basic transpositions  
$$s_k:=(k,k+1)\in S_n, \qquad(1\leq k<n).
$$ 
Note  important relations
\begin{equation}\label{E280603}
s_kL_k=L_{k+1}s_k-1,\qquad s_kL_m=L_ms_k\quad(m\neq k,k+1).
\end{equation}
The second relation immediately implies

\begin{Lemma}%\label{}%{\rm \cite{}}%{\bf ()}
Let $\bi=(i_1,\dots,i_n)\in W(n)$, and $1\leq k<n$. Then $s_kv_\bi$ is a linear combination of vectors $v_\bj$ such that $j_m=i_m$ for $m\neq k,k+1$. 
\end{Lemma}

While the role of a Cartan subalgebra is played by $A_n$, the role of  $sl_2$-subalgebras will be played by the subalgebras 
$$B_k:=\lan L_k,L_{k+1}, s_k\ran\qquad(1\leq k<n).$$ 
In view of (\ref{E280603}), every $B_k$ is a quotient of the {\em rank two degenerate affine Hecke algebra}:
$$
H_2:=\lan s,x,y\mid xy=yx, s^2=1, sx=ys-1\ran.
$$

Instead of representation theory of $sl_2$ we develop equally easy representation theory of $H_2$. We first construct some explicit $H_2$-modules. 
Fix a pair of numbers $a,b\in \mathbb{C}$. If $b=a+1$, let $L(a,b)=\mathbb{C}\cdot v$ be a $1$-dimensional vector space with the action of the generators
$$
xv=av,\ yv=bv, sv=v.
$$
Clearly, the relations are satisfied, so we have a well-defined action of $H_2$. Similarly, if $b=a-1$, we have $L(a,b)=\mathbb{C}\cdot v$ be a $1$-dimensional vector space with 
$$
xv=av,\ yv=bv, sv=-v.
$$
Finally, assume that $a\neq b\pm1$. Let $L(a,b)$ be a $2$-dimensional vector space $\C\cdot v_1\oplus \C\cdot v_2$ with the action of the generators $x,y,s$, given, respectively, by the matrices
\begin{equation}\label{E270603_2}
\left(
\begin{matrix}
a &  -1 \\
0 & b
\end{matrix}
\right),\quad
\left(
\begin{matrix}
b & 1  \\
0 & a
\end{matrix}
\right),\quad
\left(
\begin{matrix}
0 & 1  \\
1 & 0
\end{matrix}
\right).
\end{equation}
Note that if $a=b$, then $x$ and $y$ do not act on $L(a,b)$ semisimply, while if $a\neq b,b\pm1$, then we can simultaneously diagonalize $x$ and $y$, so that the  matrices of $x,y,s$ are
\begin{equation}\label{E270603_3}
\left(
\begin{matrix}
a &  0 \\
0 & b
\end{matrix}
\right),\quad
\left(
\begin{matrix}
b &  0\\
0 & a
\end{matrix}
\right),\quad
\left(
\begin{matrix}
(b-a)^{-1} & 1-(b-a)^{-2}  \\
1 & (a-b)^{-1}
\end{matrix}
\right).
\end{equation}
To achieve this, change basis from $\{v_1,v_2\}$ to $\{v_1,v_2-(b-a)^{-1}v_1\}$. If instead we change to 
\begin{equation}\label{E030703_5}
\{v_1,(1-(b-a)^{-2})^{-1/2}(v_2-(b-a)^{-1}v_1)\}, 
\end{equation}
the matrix of $s$ becomes orthogonal: 
\begin{equation}\label{E020703_5}
\left(
\begin{matrix}
(b-a)^{-1} & \sqrt{1-(b-a)^{-2}}  \\
\sqrt{1-(b-a)^{-2}} & (a-b)^{-1}
\end{matrix}
\right).
\end{equation}

It is clear that the $H_2$-modules $L(a,b)$ we have just constructed are irreducible. 
One can prove that every irreducible $H_2$-module is finite dimensional, so the finite dimensionality  assumption in the following proposition is unnecessary. 

\begin{Proposition}\label{PSimH2}%{\rm \cite{}}%{\bf ()}
We have: 
\begin{enumerate}
\item[{\rm (i)}] Every (finite dimensional) irreducible $H_2$-module is isomorphic to some $L(a,b)$.
\item[{\rm (ii)}] If $a\neq b\pm1$, then $L(a,b)\cong L(b,a)$ , and there are no other isomorphic pairs  among $\{L(a,b)\mid a,b\in\mathbb{C}\}$.
\end{enumerate} 
\end{Proposition}
\begin{proof}
(i) Let $V$ be a finite dimensional irreducible $H_2$-module. There exists $v\in V$ which is a simultaneous eigenvector for $x$ and $y$.  So
$xv=av$, $yv=bv$ for some $a,b\in \mathbb{C}$. If $sv$ is proportional to $v$, then $V=\mathbb{C}v$, and we must have $sv=\pm v$, as $s^2=1$. This immediately leads to $b=a\pm1$ and $V=L(a,b)$. If $sv$ is not proportional to $v$, then $\{v,sv\}$ must be a basis of $V$, which leads to the formulas (\ref{E270603_2}), but these formulas  determine an irreducible module only if  $a\neq b\pm1$. 

(ii) That no other pairs are isomorphic is clear, because if $L(a,b)$ and $L(c,d)$ are isomorphic, then their restrictions to subalgebras $\lan x,y\ran$, generated by $x$ and $y$, are isomorphic. Finally, if $a\neq b,b\pm 1$, it is easy to write down an explicit isomorphism between $L(a,b)$ and $L(b,a)$ using the formulas (\ref{E020703_5}).
\end{proof}

\begin{Corollary}\label{C010703}%{\rm \cite{}}%{\bf ()}
Let $\bi\in W(n)$, $V=V(\bi)$, $1\leq k<n$, and $$\bj:=s_k\bi=(i_1,\dots,i_{k-1},i_{k+1},i_k,i_{k+2},\dots,i_n).$$ Then:
\begin{enumerate}
\item[{\rm (i)}] $i_k\neq i_{k+1}$.
\item[{\rm (ii)}] If $i_{k+1}=i_k\pm 1$ then $s_k v_\bi=\pm v_\bi$ and $\bj$ is not a weight of $V$.
\item[{\rm (iii)}] Let $i_{k+1}\neq i_k\pm 1$. Then $\bj$ is a weight of $V$. Moreover, the vector
$
w:=(s_i-(i_{k+1}-i_k)^{-1})v_\bi
$
is a non-zero vector of weight $\bj$,  the elements $L_k,L_{k+1},s_k$ leave $X:=\operatorname{span}(v_\bi,w)$ invariant, and act in the basis $\{v_\bi,w\}$ of $X$ with matrices (\ref{E270603_3}), respectively.
\end{enumerate}
\end{Corollary}
\begin{proof}
By (\ref{E010703_2}), $\C S_{k+1}\cdot v_\bi\cong V(i_1,\dots,i_{k+1})$. Consider 
$$
M:=\Hom_{S_{k-1}}(V(i_1,\dots,i_{k-1}),V(i_1,\dots,i_{k+1}))
$$ 
as a module over $Z_{k-1,2}=\lan B_k,Z_{k-1}\ran$, see Proposition~\ref{POl}. This module is irreducible by Lemma~\ref{L270603}. By Schur's Lemma, $Z_{k-1}$ acts on $M$ with scalars, so $M$ is  irreducible even as a $B_k$-module. Note that the $B_k$-module $M$ is isomorphic to the $B_k$-submodule 
$$N:=B_k\cdot v_\bi\subseteq V.$$ Inflating along the surjection $H_2\to B_k$, makes $N$ into an irreducible $H_2$-module, with $i_k$ and $i_{k+1}$ appearing as eigenvalues of $x$ and $y$, respectively. Hence $N\cong L(i_k,i_{k+1})$, see Proposition~\ref{PSimH2}. 
Now the result follows from the classification of irreducible $H_2$-modules obtained above, noting for (i) that $x$ and $y$ do not act semisimply of $L(a,a)$, so this case is impossible. 
\end{proof}

%The place permutation which leads from $\bi$ to $\mu$ in Corollary~\ref{C010703} is called an {\em admissible transposition}.

\begin{Corollary}\label{C010703_3}%{\rm \cite{}}%{\bf ()}
Let $\bi=(i_1,\dots,i_n)\in \C^n$. If $i_k=i_{k+2}=i_{k+1}\pm 1$ for some $k$, then $\bi\not\in W(n)$. 
\end{Corollary}
\begin{proof}
Otherwise, Corollary~\ref{C010703}(ii) gives $s_kv_\bi=\pm v_\bi$ and $s_{k+1}v_\bi=\mp v_\bi$, which contradicts the braid relation $s_ks_{k+1}s_k=s_{k+1}s_ks_{k+1}$.  
\end{proof}

\begin{Lemma}\label{L010703_4}%{\rm \cite{}}%{\bf ()}
Let $\bi\in W(n)$. Then
\begin{enumerate}
\item[{\rm (i)}] $i_1=0$.
\item[{\rm (ii)}] $\{i_k-1,i_k+1\}\cap\{i_1,\dots,i_{k-1}\}\neq \emptyset$ for all $1<k\leq n$. 
\item[{\rm (iii)}] If $i_k=i_m=a$ for some $k<m$ then $$\{a-1,a+1\}\subseteq\{i_{k+1},\dots,i_{m-1}\}.$$ 
\end{enumerate}
\end{Lemma}
\begin{proof}
(i) is clear as $L_1=0$. 

If (ii) fails, apply Corollary~\ref{C010703}(iii) repeatedly to swap $i_k$ with $i_{k-1}$, then with $i_{k-2}$, etc., all the way to the second position. Now, if $i_k=0$, we get a weight which starts with two $0$'s, which contradicts Corollary~\ref{C010703}(i). Otherwise, again by Corollary~\ref{C010703}(iii), we can move $i_k$ to the first position, which contradicts (i). 

If (iii) fails, let us pick $k,m$ with the minimal $m-k$ for which this happens. By Corollaries~\ref{C010703}(i),(iii) and~\ref{C010703_3}, we have
$$
\bi=(\dots,a,a\pm1,\dots,a\pm1, a,\dots),
$$ 
which contradicts the minimality of $m-k$. 
%By the choice of $m-k$, there must be both $a\pm1+1$ and $a\pm1-1$ between the two entries $a\pm1$ in $\bi$. So there is an entry equal to $a$ in between, which again contradicts the minimality of $m-k$. 
\end{proof}

For any $n \geq 0$, let $\la=(\la_1,\la_2,\dots)$ be a {\em partition of $n$}, 
i.e. a weakly decreasing sequence
of non-negative integers summing to $n$. Let ${\mathcal P}(n)$ 
denote the set of all partitions
of $n$. Set 
$${\mathcal{P}} := \bigcup_{n \geq 0} {\mathcal{P}}(n).$$
We identify a partition $\la$ with its {\em Young diagram} 
$$
\la=\{(r,s)\in\Z_{>0}\times\Z_{>0}\mid s\leq\la_r\}.
$$
Elements $(r,s)\in\Z_{>0}\times\Z_{>0}$ are called {\em nodes} or {\em boxes}. 
%Recall the set $I$ from (\ref{EI}). So $I=\Z/p\Z$ if $p$ is positive, and $I=\Z$ if $p=0$.  
We label the nodes of $\la$ with {\em contents}, which are elements 
of $\Z$. 
By definition, the content of the node $(r,s)$ is $s-r$. 
The content of the node $A$ is denoted $\Res A$. 
%Define the {\em residue content} of $\la$ to be the tuple
%\begin{equation}\label{EResidueContent}
%\cont(\la)=(\gamma_i)_{i \in I},
%\end{equation}
%where for each $i \in I$, $\gamma_i$ is the number of nodes of residue $i$ contained in the diagram $\la$. 

Let $i \in \Z$.
A node $A = (r,s)
\in \la$ is called {\em $i$-removable} (resp. {\em $i$-addable}) for $\la$ 
if $\Res A=i$ and $\la_A:=\la\setminus\{A\}$ (resp. $\la^A:=\la\cup\{A\}$) is a Young diagram of a partition. 
A node is called {\em removable} (resp. {\em addable}) if it is $i$-removable (resp. $i$-addable) for some $i$. 
Thus, for example, a removable node is always of the form $(m,\la_m)$ with $\la_m>\la_{m+1}$. 

Let $\la$ be a partition of $\la$. An allocation of numbers $1,\dots,n$ into the boxes of $\la$ (one number into one box) is called a {\em $\la$-tableau}. A $\la$-tableau is called {\em standard} if the numbers increase from top to bottom along the columns of $\la$ and from left to right along the rows. 

For any $\la$-tableau $T$, let $T_k$ be the box occupied by $k$ in $T$, and  
$$\bi^T:=(\Res T_1,\dots,\Res T_n)\in\Z^n.$$ 
The symmetric group $S_n$ acts on the set of {\em all} $\la$-tableaux by acting on the entries of the tableaux. Recall that it also acts on $n$-tuples of numbers by place permutations. Then we have:
$$
w\bi^T=\bi^{w T}\qquad(w\in S_n).
$$

Define the {\em Young graph}\, $\mathbb{Y}$ as a directed graph with the set $\mathcal{P}$ of all partitions as its set of vertices; moreover, for $\la,\mu\in \mathcal{P}$ we have $\mu\to\la$ if and only if $\mu=\la_A$ for some removable node $A$ for $\la$. 
%Fix $\la\in\mathcal{P}(n)$. A {\em standard $\la$-tableau} is an allocation of numbers $1,2,\dots,n$ into the boxes of the Young diagram $\la$. 
A path in $\mathbb{Y}$ ending in $\la$ will be referred to as a  {\em $\la$-path}.
Thus a $\la$-path $T$ can be thought of as a sequence of nodes $T_1,\dots,T_n$ of $\la$ such that $T_n$ is removable for $\la$, $T_{n-1}$ is removable for $\la_{T_n}$, etc. 
%To indicate that there are $n$ nodes, we will say that $T$ has {\em length $n$}. (Clearly, an $\la$-path is the same as a standard $\la$-tableau, but we will not use this terminology). 
If, for all $1\leq k\leq n$, we place the number $k$ into the box $T_k$, we get a {\em standard $\la$-tableau}. 
In this way we get a one-to-one correspondence between  $\la$-paths in $\mathbb{Y}$ and standard $\la$-tableaux. We will not distinguish between the two.

\begin{Example}%\label{}%{\rm \cite{}}%{\bf ()}
{\rm 
If $\la=(4,2,1)$, an example of a standard $\la$-tableau is given by 
$$
%\hoogte=6pt    \newdimen\breedte   \breedte=7pt  
%\newdimen\dikte     \dikte=0.5pt 
T=\diagram{
1& 2& 4 &5 \cr
3&7\cr
6\cr
}
$$
In this case $\bi^T=(0,1,-1,2,3,-2,0)$.
}
\end{Example}

Set
\begin{equation}%\label{}
W'(n):=\{\bi^T\mid \text{$T$ is a standard $\la$-tableau for some $\la\in\mathcal{P}(n)$}\}.
\end{equation}
Note that the shape $\la$ of $T$ can be recovered from the tuple $\bi^T$: the amount of $a$'s among the $i_k$ is the amount of nodes on the $a$th diagonal of the Young diagram $\la$.  So the $n$-tuples $\bi,\bj\in W'(n)$ come from standard tableaux of the same shape if and only if $\bi$ can be obtained from $\bj$ by a place permutation, in which case we write $\bi\sim\bj$.

\begin{Lemma}\label{L010703_8}%{\rm \cite{}}%{\bf ()}
The set $W'(n)$ is precisely the set of all $n$-tuples $\bi\in\C^n$ which satisfy the properties (i)-(iii) of Lemma~\ref{L010703_4}. In particular, $W(n)\subseteq W'(n)$.
\end{Lemma}
\begin{proof}
Easy combinatorial exercise, see Exercise~\ref{ExW'}.
\end{proof}

%\begin{Lemma}%\label{}%{\rm \cite{}}%{\bf ()}
%Let $T$  and $T'$ be two paths in $\mathbb{Y}$ with the same end $\la$. Then we can get from $\bi^T$ to $\bi^T'$ by successive admissible transpositions.
%\end{Lemma}
%\begin{proof}
%It suffices to show that if $T$ is an arbitrary path in $\mathbb{Y}$ which ends in $\la=(\la_1\geq\la_2\geq\dots\geq\la_k)$
%\end{proof}

If $\bi=(i_1,\dots,i_n)\in\C^n$, and $i_k\neq i_{k+1}\pm 1$, then a place permutation which swaps $i_k$ and $i_{k+1}$ will be called an {\em admissible transposition}. If $\bi=\bi^T$ for a standard  tableau $T$, then an admissible transposition amounts to swapping $k$ and $k+1$ that do not lie on adjacent diagonals in $T$. It is clear that such a swap always transforms a standard $\la$-tableaux to a standard $\la$-tableaux. 

Let $\la=(\la_1\geq\la_2\geq\dots\geq\la_k)\in\mathcal{P}(n)$. We define the corresponding {\em canonical $\la$-tableau}
$T(\la)$ to be the $\la$-tableau obtained by filling in the numbers $1,2,\dots,n$ from left to right along the rows, starting from the first row and going down. 
%as follows: fill in the first row of $\la$ with $1,2,\dotsstart from the empty partition and add $\la_1$ nodes to fill in the first row of $\la$, then add $\la_2$ nodes to fill in the second row, and so on, finally filling in the $k$th row. 

\begin{Lemma}\label{020703}%{\rm \cite{}}%{\bf ()}
Let $\la\in\mathcal{P}(n)$. If $T$ is  a standard $\la$-tableau, then there is a  series of admissible transpositions which moves $T$ to $T(\la)$. Moreover, these transpositions $s_{k_1},s_{k_2}, \dots,s_{k_\ell}$ can be chosen in such a way that $\ell=\ell(s_{k_1}s_{k_2} \dots s_{k_\ell})$.
\end{Lemma}
\begin{proof}
Let $A$ be the last box of the last row of $\la$. In $T(\la)$, the box $A$ is occupied by $n$. In $T$, the box $A$ is occupied by some number $k $. Note also that in $T$, the numbers 
$k +1$ and $k $ do not lie on adjacent diagonals. 
So we can apply an admissible transposition to swap $k $ and $k +1$, then to swap $k +1$ and $k +2$, etc. As a result, we get a new standard $\la$-tableau in which $A$ is occupied by $n$. Next, remove $A$ together with $n$, and apply induction. Finally, note that this procedure  yields a reduced word.
%Let $T=(A_1,\dots,A_n)$, and $T(\la)=(B_1,\dots,B_n)$. Let $A_i$ be the last node of the last row of $\la$.  Then $A_{i+1}$ and $A_i$ do not lie on adjacent diagonals. So we can apply an admissible transposition to swap $A_i$ and $A_{i+1}$, then to swap $A_{i}$ and $A_{i+2}$, etc. As a result, we get a new path $(A_1',\dots,A_n')$ such that $A_n'=B_n$. Next, take care of the $(n-1)$st position, etc. Note that the procedure yields a reduced word.
\end{proof}

\begin{Lemma}\label{L010703_9}%{\rm \cite{}}%{\bf ()}
If $\bi\in W'(n)$ and $\bi\sim\bj$ for some $\bj\in W(n)$, then $\bi\in W(n)$ and $\bi\approx\bj$.
\end{Lemma}
\begin{proof}
%If $\nu=(\nu_1,\dots,\nu_n)\in\C^n$, and $\nu_i\neq \nu_{i+1}\pm 1$, then a place permutation which swaps $\nu_i$ and $\nu_{i+1}$ will be called an {\em admissible transposition}. If $\nu=\bi^T$ for a path $T=(A_1,\dots,A_n)$, then an admissible transposition amounts to swapping $A_i$ and $A_{i+1}$ that do not lie on adjacent diagonals. It is clear that such a swap always moves a path to a path.
By definition, $\bi=\bi^T$ for some standard tableau $T$.  
By Lemma~\ref{L010703_8}, $\bj=\bi^S$. As $\bi\sim\bj$, the tableaux $S$ and $T$ have the same shape. In view of Corollary~\ref{C010703}(iii), it suffices to show that we can go from $\bi^S$ to $\bi^T$ by a chain of admissible transpositions. But this follows from Lemma~\ref{020703}.
%Let $T_0$ be the path obtained as follows. Start from the empty partition and add $\la_1$ nodes to fill in the first row of $\la$, then add $\al_1$ nodes to fill in the second row, and so on, finally filling in the $k$th row. It now suffices to prove that if $T$ is an arbitrary path in $\mathbb{Y}$ which ends in $\al$, then by a series of admissible transpositions we can move $\bi^T$ to $\bi_{T_0}$.
%To achieve that, let $T=(A_1,\dots,A_n)$, and $T_0=(C_1,\dots,C_n)$. The last node of the last row of $\al$ must be one of these nodes, $A_i$ say.  It is clear that $A_{i+1}$ is not on the adjacent diagonal with that of $A_i$. So we can apply an admissible transposition to swap $A_i$ and $A_{i+1}$, then to swap $A_{i+1}$ and $A_{i+2}$, etc. As a result we will have a new path $(B_1,\dots,B_n)$ obtained from $T$ by a series of admissible transpositions and such that $B_n=C_n$. Now apply induction on $n$.
\end{proof}

%The following is the main result of this chapter. 

\begin{Theorem}\label{TOV}%{\rm \cite{}}%{\bf ()}
We have
$
W(n)=W'(n).
$
Moreover, $\bi^T\approx\bi^{S}$ if and only if 
%$\equivalent$ 
%the tableaux $T$ and $S$ have the same shape $\equivalent$ 
$\bi^T\sim\bi^{S}$. In particular, the branching graph $\mathbb{B}$ is isomorphic to the Young graph $\mathbb{Y}$. 
\end{Theorem}
\begin{proof}
By Lemma~\ref{L010703_8}, $W(n)\subseteq W'(n)$. The number of isomorphism classes of irreducible $\C S_n$-modules equals the number of conjugacy classes of $S_n$, which are labelled by partitions of $n$, see Exercise~\ref{ExConjCl}. So we have 
\begin{equation}\label{E010703_9}
|W(n)/\approx|=|\mathcal{P}(n)|=|W'(n)/\sim|. 
\end{equation}
Now, let $\bi\in W'(n)$. In view of Lemma~\ref{L010703_9},  the $\sim$-equivalence class of $\bi$ either contains no elements of $W(n)$ or is a subset of a $\approx$-equivalence class of $W(n)$. In view of (\ref{E010703_9}), this now implies $W(n)=W'(n)$ and $\sim$ is equivalent to $\approx$. 
\end{proof}

Now to every irreducible $\C S_n$-module $V$ we can associate a partition $\la\in\mathcal{P}(n)$. Indeed, if $\bi\in W(n)$ is a weight of $V$ then $\bi=\bi^T$ for some standard tableaux $T$, and we associate to $V$ the shape $\la$ of $T$, which is well-defined by the theorem. We will write $V=V^\la$. This notation is better than $V(\bi)$, because we have a one-to-one correspondence between the isoclasses of irreducible $\C S_n$-modules and partitions of $n$. The weights of $V^\la$ are precisely $\{\bi^T\mid \text{$T$ is a standard $\la$-tableau}\}$.

\begin{Example}\label{E050307}%{\rm \cite{}}%{\bf ()}
{\rm 
(i) If $\la=(n)$, the only standard $\la$-tableau is\ \ 
$
\begin{picture}(60,10) 
\put(0,0){\line(1,0){35}}
\put(0,10){\line(1,0){35}}
\put(0,0){\line(0,1){10}}
\put(10,0){\line(0,1){10}}
\put(20,0){\line(0,1){10}}
\put(30,0){\line(0,1){10}}
%\put(40,0){\line(0,1){10}}
\put(40,5){\makebox(0,0){$\cdots$}}
\put(45,0){\line(1,0){15}}
\put(45,10){\line(1,0){15}}
\put(50,0){\line(0,1){10}}
\put(60,0){\line(0,1){10}}
\put(5,5){\makebox(0,0){\scriptsize$1$}}
\put(15,5){\makebox(0,0){\scriptsize$2$}}
\put(25,5){\makebox(0,0){\scriptsize$3$}}
\put(55,5){\makebox(0,0){\scriptsize$n$}}
\end{picture}
$\,.
So $V^{(n)}$ is $1$-dimensional, and its only weight is $(0,1,\dots,n-1)$. Similarly, $V^{(1^n)}$ is $1$-dimensional with the only weight $(0,-1,\dots,-n)$. It is clear from this information that $V^{(n)}$ is the trivial and $V^{(1^n)}$ is the sign modules over $S_n$.

(ii) Let $\la=(n-1,1)$. Then the standard $\la$-tableaux are 
$T(k):=
\begin{picture}(52,20) 
\put(0,5){\line(1,0){25}}
\put(0,15){\line(1,0){25}}
\put(0,5){\line(0,1){10}}
\put(10,5){\line(0,1){10}}
\put(20,5){\line(0,1){10}}
%\put(30,5){\line(0,1){10}}
%\put(40,0){\line(0,1){10}}
\put(30,10){\makebox(0,0){$\cdots$}}
\put(35,5){\line(1,0){15}}
\put(35,15){\line(1,0){15}}
\put(40,5){\line(0,1){10}}
\put(50,5){\line(0,1){10}}
\put(5,10){\makebox(0,0){\scriptsize$1$}}
\put(15,10){\makebox(0,0){\scriptsize$2$}}
%\put(25,10){\makebox(0,0){\scriptsize$3$}}
\put(45,10){\makebox(0,0){\scriptsize$n$}}
\put(0,5){\line(0,-1){10}}
\put(0,-5){\line(1,0){10}}
\put(10,-5){\line(0,1){10}}
\put(5,0){\makebox(0,0){\scriptsize$k$}}
\end{picture}
$
for $2\leq k\leq n$, and the corresponding weights are $$\bi^{(k)}:=(0,1,\dots,k-2,-1,k-1,\dots,n-2)\qquad (2\leq k\leq n).$$ 
%Then the weights of $V^{(n-1,1)}$ are precisely $\{\bi(i)\mid 2\leq i\leq n\}$.
}
\end{Example}

Note what we have done so far. We have started from a nested family of algebras $\C S_0\subset \C S_1\subset  \dots$, proved the multiplicity-freeness of the branching rule from scratch, defined the branching graph $\mathbb{B}$, and tried to learn enough facts about $\mathbb{B}$, so that we could identify it with some known graph. This have lead to a classification of irreducible $\C S_n$-modules for all $n$ and a description of the branching rule at the same time. On the way we have obtained other useful results about irreducible modules.  %We are going to follow this scheme again and again in this book for various families of algebras, although to realize it we will need more sofisticated tools. 

\section{Formulas of Young and Murnaghan-Nakayama}\label{S2.2}

Formulas of Young describe explicitly the matrices of simple transpositions $s_k $ with respect to a nice choice of a $GZ$-basis. The formulas come more or less from (\ref{E270603_3}) and (\ref{E020703_5}). We just need to scale the elements of a GZ-basis in a consistent way. 

In order to do this, fix $\la\in\mathcal{P}(n)$. Pick a basis  vector $v_{T(\la)}\in V_{T(\la)}^\la$ corresponding to the canonical $\la$-tableau. Let $T$ be an arbitrary standard $\la$-tableau. Write $T=w\cdot T(\la)$ for $w\in S_n$. Define $\ell(T)$ to be $\ell(w)$. Denote by $\pi_T$ the projection to the one-dimensional subspace $V_T^\la$ along $\oplus_{S\neq T}V_S^\la$, and set 
\begin{equation}\label{E030703}
v_T=\pi_T(wv_{T(\la)}).
\end{equation}
By Lemma~\ref{020703}, there is a reduced decomposition $w=s_{k _1}\dots s_{k _\ell}$ with all simple transpositions being admissible. So Corollary~\ref{C010703}(iii) implies
\begin{equation}\label{E030703_2}
wv_{T(\la)}=v_T+\sum_{S:\,\ell(S)<\ell(T)} c_S v_S,
\end{equation}
and $v_T\neq 0$.

\begin{Theorem}\label{TYSN}%{\rm \cite{}}%
{\bf (Young's Seminormal Form)}
Let $\la\in\mathcal{P}(n)$, $\{v_T\}$ be the GZ-basis of $V^\la$ defined in (\ref{E030703}), and $1\leq k<n$. Then the action of the simple transposition $s_k\in S_n$  is given as follows
\begin{enumerate}
\item[{\rm (i)}] If $\Res T_{k+1}=\Res T_{k}\pm 1$, then $s_kv_T=\pm v_T$. 
\item[{\rm (ii)}] Let $\rho:=(\Res T_{k+1}-\Res T_{k})^{-1}\neq \pm 1$ and set $S=s_kT$. Then
$$
s_kv_T=\begin{cases}
\rho v_T+v_S, 
& \text{if $\ell(S)>\ell(T)$,}\\
-\rho v_T+(1-\rho^2) v_S, &\text{if $\ell(S)<\ell(T)$.}
\end{cases}
$$
\end{enumerate}
\end{Theorem}
\begin{proof}
If $\Res T_{k+1}=\Res T_{k}\pm 1$, the result follows from Corollary~\ref{C010703}(ii). Otherwise $s_k$ is an admissible transposition for $T$. We may assume that $\ell(S)>\ell(T)$. As weight  spaces of $V^\la$ are $1$-dimensional, Corollary~\ref{C010703}(iii) implies that $v_S$ equals $s_kv_T-\rho v_T$ up to a scalar multiple, and, using (\ref{E030703_2}), we see that the scalar is $1$. 
%Now the result follows from  Corollary~\ref{C010703}(iii). 
\end{proof}

\begin{Corollary}%\label{}%{\rm \cite{}}%{\bf ()}
Irreducible representations of $S_n$ are defined over $\mathbb{Q}$ and are self-dual.
\end{Corollary}

\begin{Theorem}\label{TYOF}%{\rm \cite{}}%
{\bf (Young's Orthogonal Form)}
Let $\la\in\mathcal{P}(n)$. There exists a GZ-basis $\{w_T\}$ of $V^\la$ such that the action of an arbitrary simple transposition $s_k\in S_n$  is given by
\begin{equation}\label{E130613}
s_kw_T=\rho w_T+\sqrt{1-\rho^2}w_{s_kT},
\end{equation}
where $\rho:=(\Res T_{k+1}-\Res T_{k})^{-1}$ (note that when $\rho=\pm1$, the coefficient of $w_{s_kT}$ is zero, so this term should be omitted).
\end{Theorem}
\begin{proof}
Let $\{v_T\}$ be the basis of Theorem~\ref{TYSN}, and set $$w_T=v_T/\sqrt{(v_T,v_T)}.$$  Let $S=s_kT$. We may assume that $s_k$ is an admissible transposition. Moreover, note that the formula (\ref{E130613}) for $s_kw_T$ implies the corresponding formula for $s_kw_S$ and conversely, so  we may assume that $\ell(S)>\ell(T)$.

As $s_k$ preserves $(\cdot,\cdot)$, the formulas of Theorem~\ref{TYSN}(ii)
imply 
\begin{align*}
(v_{S},v_{S})&=(s_i v_T-\rho v_T,s_i v_T-\rho v_T)
\\
&=(v_T,v_T)+\rho^2(v_T,v_T)-\rho(s_iv_T,v_T)-\rho(v_T,s_iv_T)
\\
&=(v_T,v_T)+\rho^2(v_T,v_T)-\rho(\rho v_T+v_S,v_T)-\rho(v_T,\rho v_T+v_S)
\\
&=(1-\rho^2)(v_T,v_T).
\end{align*}
Hence 
$$w_S=v_S/(v_S,v_S)=v_S/(\sqrt{(v_T,v_T)}\sqrt{1-\rho^2}).
$$ 
Now, the result follows from (\ref{E030703_5}) and (\ref{E020703_5}). 
\end{proof}

\begin{Example}%\label{}%{\rm \cite{}}%{\bf ()}
{\rm 
Let $\la=(n-1,n)$. Using the notation of Example~\ref{E050307}(ii) and writing $v_j$ for $v_{T(j)}$, $2\leq j\leq n$, the formulas of Young's orthogonal form become:
\begin{equation}\label{E050703_4}
s_iv_{j}= \begin{cases}
v_j, & \text{if $j\neq i,i+1$,}\\
\frac{1}{i}v_i+\sqrt{1-\frac{1}{i^2}}v_{i+1}, & \text{if $j= i$,}\\
\sqrt{1-\frac{1}{i^2}}v_i-\frac{1}{i}v_{i+1}, & \text{if $j= i+1$.}
\end{cases}
\end{equation}
Let $M$ be the natural permutation $\C S_n$-module with basis $e_1,\dots,e_n$. It has the irreducible submodule $N=\{\sum_ia_ie_i\in M\mid \sum_i a_i=0\}$. Set $$v_j:=\frac{1}{\sqrt{j(j-1)}}(e_1+\dots+e_{j-1}-(j-1)e_j)\qquad(2\leq j\leq n).$$ Then $\{v_2,v_3,\dots,v_n\}$ is a basis of $N$ with respect to which the simple permutations act by formulas (\ref{E050703_4}). 
}
\end{Example}

Let $\la\in \mathcal{P}(n)$ and $\mu\in\mathcal{P}(n-k)$. Set
$$
V^{\la/\mu}:=\Hom_{S_{n-k}}(V^\mu,\res_{S_{n-k}} V^\la).
$$
It is clear from the branching rule that $V^{\la/\mu}\neq 0$ if and only if the Young diagram $\mu$ is contained in the Young diagram $\la$, in which case we denote  the complement by $\la/\mu$. A set of nodes of this form will be called a {\em skew shape}. The number of nodes in $\la/\mu$ will be denoted $|\la/\mu|$. The number of rows occupied by $\la/\mu$ minus $1$ will be denoted by $L(\la/\mu)$. 
A skew shape is called a {\em skew hook} if it is connected and does not have two boxes on the same diagonal (equivalently, if the contents of the nodes of the shape form a segment of integers). 

By Lemma~\ref{L270603}, we know that $V^{\la/\mu}$ is an irreducible $Z_{n-k,k}$-module. On restriction to $S_k'\subset Z_{n-k,k}$ it becomes a (not necessarily irreducible) $\C S_k$-module. Let $\chi^{\la/\mu}$ be the character of this $\C S_k$-module. If $\mu=\emptyset$ we get the character $\chi^\la$ of $V^\la$. 
The results on GZ-bases and Young's canonical forms can be easily generalized to skew shapes. For example, define a {\em $\la/\mu$-path} to be any path which connects $\mu$ with $\la$. We will not distinguish between $\la/\mu$-paths and standard $\la/\mu$-tableaux (defined in the obvious way). 
Then Theorem~\ref{TYOF} implies

\begin{Proposition}\label{PYOFSS}%{\rm \cite{}}%
{\bf (Young's Orthogonal Form for Skew Shapes)}
Let $\la/\mu$ be a skew shape with $|\la/\mu|=k$, $|\mu|=n-k$. There exists a basis 
$$\{w_T\mid \text{$T$ is a standard $\la/\mu$-tableau}\}
$$ 
of $V^{\la/\mu}$ such that the action of an arbitrary simple transposition $s_r\in S_k$ is given by
$$
s_rw_T=\rho w_T+\sqrt{1-\rho^2}w_{s_rT}
$$
where $\rho:=(\Res T_{r+1}-\Res T_{r})^{-1}$ (note that when $\rho=\pm1$, the coefficient of $w_{s_rT}$ is zero, so this term should be omitted). Moreover, each vector $w_T$ is a simultaneous eigenvector for $L_{n-k+1},\dots,L_n\in Z_{n-k,k}$ with eigenvalues $\Res T_1,\dots,\Res T_k$, respectively. 
\end{Proposition}

\begin{Lemma}\label{L040703}%{\rm \cite{}}%{\bf ()}
Let $\la/\mu$ be a skew shape with $|\la/\mu|=k$,  and $T$ be a standard $\la/\mu$-tableau. Then $\C S_k\cdot w_T=V^{\la/\mu}$.
\end{Lemma}
\begin{proof}
As $V^{\la/\mu}$ is irreducible over $ Z_{n-k,k}$, we have $ Z_{n-k,k}\cdot w_T=V^{\la/\mu}$. On the other hand, in view of Olshanskii's Lemma and (\ref{E280603}), every element of $ Z_{n-k,k}$ can be written as $gxz$, where $g\in \C S_k$, $x\in\lan L_{n-k+1},\dots,L_n\ran$, $z\in  Z_{n-k}$. As $x$ and $z$ act on $w_T$ by multiplication with scalars, %(see Proposition~\ref{PYOFSS}), 
the result follows. 
\end{proof}

\begin{Lemma} \label{LDisj}%{\rm \cite{}}%{\bf ()}
Let $\la/\mu=\ga\cup\de$ where $\ga$ and $\de$ are skew shapes disconnected from each other. Let $c:=|\ga|$ and $d:=|\de|$. Then, as $S_{c+d}$-modules,  
$$V^{\la/\mu}\cong \ind^{S_{c+d}}_{S_c\times S_d}(V^\ga\boxtimes V^\de).$$ 
\end{Lemma}
\begin{proof}
There exists a standard $\la/\mu$-tableau $T$ such that $T_1,\dots,T_c\in \ga$ and $T_{c+1},\dots,T_k\in\de$. By Proposition~\ref{PYOFSS}, the subspace of $V^{\la/\mu}$, spanned by vectors $w_T$ for all such tableaux $T$, is invariant with respect to $S_c\times S_d<S_k$, and, as a $\C[S_c\times S_d]$-module, it is isomorphic to $V^\ga\boxtimes V^\de$. By Lemma~\ref{L040703} and Frobenius reciprocity, we get a surjective homomorphism 
$$\ind^{S_{c+d}}_{S_c\times S_d}(V^\ga\boxtimes V^\de)\to V^{\la/\mu}.$$ 
But, using Proposition~\ref{PYOFSS}, we see that the dimensions of both modules are equal to ${k\choose c}\dim V^\ga\dim V^\de$. So $V^{\la/\mu}\cong \ind^{S_{c+d}}_{S_c\times S_d}(V^\ga\boxtimes V^\de)$. 
\end{proof}
%The proposition implies that $V^{\la/\be}\cong V^{\la'/\be'}$ if $\la/\be$ and $\la'/\be'$ `have the same shape'.

The final main result of this section is 

\begin{Theorem}\label{TMNG}%{\rm \cite{}}%{\bf ()}
Let $\la/\mu$ be a skew shape with $|\la/\mu|=k$. Then 
$$
\chi^{\la/\mu}\big((1,2,\dots,k)\big)=\begin{cases}
(-1)^{L(\la/\mu)}, 
& \text{if $\la/\mu$ is a skew hook,}\\
0, &\text{otherwise.}
\end{cases}
$$
\end{Theorem}

Before proving Theorem~\ref{TMNG}, we note the following corollary, which provides us with a very effective way to evaluate an irreducible character on a given element.

\begin{Corollary}\label{CMN}%{\rm \cite{}}%
{\bf (Murnaghan-Nakayama Rule)} Let $\la/\mu$ be a skew shape with $|\la/\mu|=k$, and $c$ be an element of $S_k$ whose cycle shape is $\rho=(\rho_1,\dots, \rho_l)$. Then
$$
\chi^{\la/\mu}(c)=\sum_H (-1)^{L(H)},
$$
where the sum is over all sequences $H$ of partitions
$$
\mu=\la(0)\subset \la(1)\subset\dots\subset\la(l)=\la
$$
such that $\la(i)/\la(i-1)$ is a skew hook with $|\la(i)/\la(i-1)|=\rho_i$ for all $1\leq i\leq l$, and $L(H)=\sum_{i=1}^l L(\la(i)/\la(i-1))$.  
\end{Corollary}
\begin{proof}
By the branching rule, for $m<k$ we have 
$$
\res_{S_{m}\times S_{k-m}} V^{\la/\mu}=\oplus_{\nu}V^{\nu/\mu}\boxtimes V^{\la/\nu},
$$ 
where the sum is over all partitions $\nu$ with $\mu\subset\nu\subset\la$ such that $|\nu/\mu|=m$. More generally,
$$
\res_{S_{\rho_1}\times \dots\times S_{\rho_l}}V^{\la/\mu}=\bigoplus_{\mu=\la(0)\subset \la(1)\subset\dots\subset\la(l)=\la}V^{\la(1)/\la(0)}\boxtimes \dots\boxtimes V^{\la(l)/\la(l-1)},
$$ 
Now the result follows from Theorem~\ref{TMNG}.
\end{proof}

We proceed to prove Theorem~\ref{TMNG}. Fix a skew shape $\la/\mu$ with $|\la/\mu|=k$. 

\begin{Lemma}\label{L040703_2}%{\rm \cite{}}%{\bf ()}
Theorem~\ref{TMNG} is true for $\mu=\emptyset$.
\end{Lemma}
\begin{proof}
It is easy to see that $L_2L_3\dots L_k$ is the sum of all $k$-cycles in $S_k$. If $v\in V^\la$ is a weight vector of weight $\bi$, then $L_2L_3\dots L_kv=i_2i_3\dots i_nv$, which is zero unless $\la$ is a hook, see Theorem~\ref{TOV}. On the other hand, if $\la=(k-b,1^b)$ is a hook with $L(\la)=b$, then, again by Theorem~\ref{TOV}, we have $i_2\dots i_n=(-1)^bb!(k-b-1)!$ and $\dim V^\la=\binom{k-1}{b}$.  Now the result follows from the fact that there are $(k-1)!$\, $k$-cycles in $S_k$. 
\end{proof}

\begin{Lemma}\label{L050703}%{\rm \cite{}}%{\bf ()}
If $\la/\mu$ is not connected, then $\chi^{\la/\mu}\big((1,2,\dots,k)\big)=0$.
\end{Lemma}
\begin{proof}
Let $\la/\mu=\ga\cup\de$ where $\ga$ and $\de$ are skew shapes disconnected from each other. Let $c:=|\ga|$ and $d:=|\de|$. 
By Lemma~\ref{LDisj}, we have   
$V^{\la/\mu}\cong \ind^{S_{c+d}}_{S_c\times S_d}(V^\ga\boxtimes V^\de).$ 
Now the lemma follows from the following standard general fact: if $H$ is a subgroup of a finite group $G$, $g\in G$ is not conjugate to an element of $H$, and $V$ is a $\C G$-module induced from $H$, then the character of $V$ on $g$ is zero, cf. Exercise~\ref{ExZeroChar}
\end{proof}

\begin{Lemma}%\label{}%{\rm \cite{}}%{\bf ()}
If $\la/\mu$ has two nodes on the same diagonal, and $\nu=(a,1^{k-a})$ be an aritrary hook with $k$-boxes, then $V^\nu$ is not a composition factor of $V^{\la/\mu}$. In particular,  $\chi^{\la/\mu}\big((1,2,\dots,k)\big)=0.$
\end{Lemma}
\begin{proof}
The second statement follows from the first by Lemma~\ref{L040703_2}. 
By assumption a $2\times 2$ square\, 
$
\hoogte=6pt    \newdimen\breedte   \breedte=7pt  
\newdimen\dikte     \dikte=0.5pt 
\diagram{
&  \cr
&\cr
}
$\, is contained in $\la/\mu$. It follows from Proposition~\ref{PYOFSS} that $V^{(2,2)}$  is an $S_4$-submodule of $V^{\la/\mu}$ (for $S_4$ embedded not necessarily with respect to the first $4$ letters, but such $S_4$ is conjugate to the canonical one anyway). By Frobenius reciprocity and Lemma~\ref{L040703}, there is a surjection $\ind^{S_k}_{S_4}V^{(2,2)}\to V^{\la/\mu}$, and the result now follows from the branching rule. 
\end{proof}

\begin{Lemma}\label{L050703_2}%{\rm \cite{}}%{\bf ()}
Let $\la/\mu$ be a skew hook, and $\nu=(k-b,1^b)$. Then $V^\nu$ appears as a composition factor of $V^{\la/\mu}$ if and only if $b=L(\la/\mu)$, in which case its multiplicity is one. 
\end{Lemma}
\begin{proof}
It follows from Proposition~\ref{PYOFSS} that translation of $\la/\mu$ does not change the corresponding $S_k$-module. So we may assume that $\la$ and $\mu$ are minimal possible, as in the picture
$$
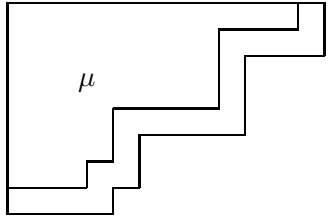
\begin{picture}(120,80) 
\put(0,0){\line(0,1){80}}
\put(0,80){\line(1,0){120}}
\put(120,80){\line(0,-1){20}}
\put(120,60){\line(-1,0){30}}
\put(90,60){\line(0,-1){30}}
\put(90,30){\line(-1,0){40}}
\put(50,30){\line(0,-1){20}}
\put(50,10){\line(-1,0){10}}
\put(40,10){\line(0,-1){10}}
\put(40,0){\line(-1,0){40}}
\put(110,80){\line(0,-1){10}}
\put(110,70){\line(-1,0){30}}
\put(80,70){\line(0,-1){30}}
\put(80,40){\line(-1,0){40}}
\put(40,40){\line(0,-1){20}}
\put(40,20){\line(-1,0){10}}
\put(30,20){\line(0,-1){10}}
\put(30,10){\line(-1,0){30}}
\put(30,50){\makebox(0,0){$\mu$}}
\end{picture}\\
$$
Now, if $b\neq L(\la/\mu)$ then $\nu\not\subseteq\la$, so by the branching rule, $V^\nu$ does not appear as a composition factor of $\res_{S_k}V^\la$, hence it does not appear in $V^{\la/\mu}$ either. 

Let $b=L(\la/\mu)$. Note that $\la/\nu$ has shape $\mu$. So it follows from Proposition~\ref{PYOFSS} that $V^{\la/\nu}$ and $V^\mu$ are isomorphic as $\C S_{n-k}$-modules. So $[\res_{S_{k}\times S_{n-k}}V^\la:V^\nu\boxtimes V^\mu]=1$. Then $[\res_{S_{n-k}\times S_{k}}V^\la:V^\mu\boxtimes V^\nu]=1$. It remains to note that $[V^{\la/\mu}:V^\nu]=[\res_{S_{n-k}\times S_{k}}V^\la:V^\mu\boxtimes V^\nu]$. 
\end{proof}

Theorem~\ref{TMNG} follows from Lemmas~\ref{L040703_2} and~\ref{L050703}-\ref{L050703_2}.

\begin{Remark}%\label{}%{\rm \cite{}}%{\bf ()}
{\rm 
We sketch another interpretation of the graph $\mathbb{Y}$. Let $\mathfrak{g}= \mathfrak{gl}_\infty(\C)$ be the Lie algebra of all $\Z\times\Z$-matrices over $\C$ with only finitely many non-zero entries. Thus, the matrix units $\{E_{ij}\mid i,j\in \Z\}$ form a basis of $\mathfrak{g}$. The Lie algebra $\mathfrak{g}$ acts on the {\em  fermionic Fock space} $\mathcal{F}$, which is the complex vector space, whose basis consists of the formal semi-infinite wedges $v_{i_0}\wedge v_{i_{1}}\wedge v_{i_{2}}\wedge\cdots$ such that $i_0>i_1>\dots$ and $i_k=-k$ for $k\gg 0$. To write down the action we follow the usual rules for the action of Lie algebra on a wedge power of a module. For example, 
$$
E_{2,-1}\cdot v_0\wedge v_{-1}\wedge v_{-2}\wedge\cdots=
v_0\wedge v_{2}\wedge v_{-2}\wedge\cdots = -v_2\wedge v_{0}\wedge v_{-2}\wedge\cdots.
$$
In fact, more than just $\mathfrak{g}$ acts on $\mathcal{F}$. Let $\La_k=\sum_{j-i=k}E_{i,j}$ be the $k$th diagonal. Even though $\La_k$ is not an element of $\mathfrak{g}$, we can still extend the action of $\mathfrak{g}$ to it, at least if $k\neq 0$. For example,
$$
\La_{-2} \cdot v_0\wedge v_{-1}\wedge v_{-2}\wedge\cdots=
v_2\wedge v_{-1}\wedge v_{-2}\wedge\cdots -v_1\wedge v_{0}\wedge v_{-2}\wedge\cdots.
$$

It is convenient to label semi-infinite wedges by partitions: to a partition $\la=(\la_1\geq\la_2\geq\dots)$ we associate the  vector $v_\la:=v_{\la_1}\wedge v_{\la_2-1}\wedge v_{\la_3-2}\wedge\cdots$.  For example, $v_\emptyset=v_0\wedge v_{-1}\wedge v_{-2}\wedge\cdots$. 
Then $\{v_\la\mid\la\in\mathcal{P}\}$ is a basis of $\mathcal{F}$, and we have in some sense recovered the vertices of $\mathbb{Y}$. For the edges, note that $E_{i,i+1}v_\la=v_\mu$ where $\mu$ is obtained from $\la$ by removing a removable node of content $i$, if it exists, and otherwise $v_\mu$ is interpreted as $0$. Similarly, $E_{i+1,i}v_\la=v_\nu$ where $\nu$ is obtained from $\la$ by adding an addable node of content $i$, if it exists, and otherwise $v_\nu$ is interpreted as $0$. Thus the action of the Chevalley generators of $\mathfrak{g}$ on the basis vectors $\{v_\la\}$ recovers the edges of $\mathbb{Y}$. Is it possible to explain this remarkable coincidence of two graphs, one coming from representation theory of $S_n$ and the other from (completely different) representation theory of $\mathfrak{gl}_\infty(\C)$? ... 

We make one more observation along these lines. It is easy to see that for $i<j$ we have $E_{i,j}v_\la=\eps v_\mu$, where $\mu$ is obtained from $\la$ by removing a skew hook of length $j-i$, starting at the node of content $i$ and ending at the node of content $j-1$; if no such hook exists, interpret $v_\mu$ as $0$. Moreover, $\eps=(-1)^{L(\la/\mu)}$. It follows that for $k>0$ we have 
$$
\La_k v_\la=\sum (-1)^{L(\la/\mu)} v_\mu
$$
where the sum is over all $\mu$ such that $\la/\mu$ is a skew hook with $|\la/\mu|=k$. So the Murnaghan-Nakayama rule can be interpreted as follows:  the value $\chi^\la(c_\rho)$ of the irreducible character $\chi^\la$ on an element $c_\rho$ with cycle-shape $(\rho_1,\rho_2,\dots,\rho_l)$ is equal to the coefficient of $v_\emptyset$ in $\La_{\rho_1}\La_{\rho_2}\dots \La_{\rho_l} v_\la$.  Or better yet:
\begin{equation}
\chi^\la(c_\rho)=(v_\la\,,\,\La_{-\rho_1}\La_{-\rho_2}\dots \La_{-\rho_l} v_\emptyset),
\end{equation}
where $(\cdot,\cdot)$ is the contravariant form on $\mathcal{F}$ normalized so that $(v_\emptyset,v_\emptyset)=1$. In fact, the form $(\cdot,\cdot)$ is determined from
$$
(v_\la,v_\mu)=\de_{\la,\mu}\qquad(\la,\mu\in\Par). 
$$
}
\end{Remark}

\chapter{Day Three}

\section{Heisenberg algebra and Boson-Fermion correspondence}
Recall the operators 
$$
\La_k=\sum_{i\in\Z}E_{i,i+k}\qquad(k\in\Z).
$$
These are linear operators on the infinite dimensional vector space
$$
V:=\bigoplus_{i\in \Z}\C\cdot v_i. 
$$
It is clear that these linear operators commute. So we can consider $V$ as a representation of the `silly Lie algebra'
$$
\a=\bigoplus_{k\in\Z}\C\cdot a_k,
$$
with commutation relations $[a_k,a_m]=0$ for all $k,m\in\Z$. In this representation, we map $a_k\mapsto\La_k$. 

We have noticed that the $\La_k$ also act on the Fock space $\F$ when $k\neq 0$. On the other hand the action of $\La_0$ is not well defined---it leads to a computation of an infinite sum. So we will force $a_0$ to act on $\F$ as zero. But now there is another problem: on $\F$ the operators $\La_k$ do not quite commute:

\begin{Lemma} \label{LLaComm}%{\rm \cite{}}%{\bf ()}
For any $m,n\in\Z$, we have 
$$
[\La_n,\La_k]=n\de_{n,-k}\id.
$$
\end{Lemma}
\begin{proof}
If $n\neq -k$, then 
\begin{align*}
[\La_n,\La_k]&=(\sum_{i\in\Z}E_{i,i+n})(\sum_{j\in\Z}E_{j,j+k})-(\sum_{j\in\Z}E_{j,j+k})(\sum_{i\in\Z}E_{i,i+n})
\\
&=\sum_{i,j\in\Z}[E_{i,i+n},E_{j,j+k}]
\\
&=\sum_{i\in\Z}E_{i,i+n+k}-\sum_{j\in\Z}E_{j,j+k+n}=\La_{n+k}-\La_{n+k}=0.
\end{align*}

To deal with the commutator $[\La_k,\La_{-k}]$, we first observe that
\begin{equation}\label{EELa}
[E_{i,j},\La_k]=E_{i,j}\sum_{n\in\Z}E_{n,n+k}-\sum_{n\in\Z}E_{n,n+k}E_{i,j}=E_{i,j+k}-E_{i-k,j}.
\end{equation}
Now, by Leibnitz formula, we have
\begin{align*}
&[E_{i,j},[\La_k,\La_{-k}]]
\\=&[[E_{i,j},\La_k],\La_{-k}]+[\La_k,[E_{i,j},\La_{-k}]]
\\
=&[E_{i,j+k}-E_{i-k,j},\La_{-k}]+[\La_k,E_{i,j-k}-E_{i+k,j}]
\\
=&E_{i,j}-E_{i+k,j+k}-E_{i-k,j-k}+E_{i,j}-E_{i,j}+E_{i-k,j-k}+E_{i+k,j+k}-E_{i,j}=0.
\end{align*}
Thus $[\La_k,\La_{-k}]$ commutes with all $E_{i,j}$. Since we can obtain any basis vector $v_\la$ from $v_\emptyset$ by an application of several $E_{i,j}$, it follows that it suffices to prove that 
$$
[\La_k,\La_{-k}]v_\emptyset=kv_\emptyset.
$$
We may assume that $k>0$. Then $\La_kv_\emptyset=0$, and so 
$[\La_k,\La_{-k}]v_\emptyset=\La_k\La_{-k}v_\emptyset$. Now, $\La_{-k}=\sum \pm v_\chi$, where the sum is over all hooks $\chi$ of size $k$. Since there are exactly $k$ such hooks and $\La_k$ `undoes' the hooks, the result follows.
\end{proof}

It follows from the lemma that the fermionic Fock space $\F$ can actually be considered as a representation of the {\em Heisenberg algebra}:
$$
H=\bigoplus_{k\in \Z}\C\cdot a_k\oplus \C\cdot z
$$
with commutation relations
$$
[z,a_k]=0,\quad[a_n,a_k]=n\de_{n,-k}z.
$$

%Note that the fermionic Fock space $\F$ is a representation of $H$. 
There is another important representation of $H$, which is called the {\em bosonic Fock space} $\B$. As a vector space, 
$$
\B=\C[x_1,x_2,\dots],
$$ 
the space of polynomials in infinitely many variables $x_1, x_2,\dots$. 
Given $\al,\zeta\in\C$, define the representation $\B(\al,\zeta)$ of $H$ on $\B$:
\begin{align*}
a_n&=\partial/\partial x_n\qquad(n\in\Z_{>0});
\\
a_{-n}&=\zeta n x_n\qquad(n\in\Z_{>0});
\\
a_{0}&=\al\id;
\\
z&=\zeta\id.
\end{align*}

\begin{Lemma} %\label{}%{\rm \cite{}}%{\bf ()}
If $\zeta\neq 0$, then the representation $\B(\al,\zeta)$ is irreducible.
\end{Lemma}
\begin{proof}
Any polynomial in $\B$ can be reduced to a multiple of $1$ by successive application of the $a_n$ with $n > 0$. Then successive
application of the $a_n$ with $n < 0$ can give any monomial in $\B$.
\end{proof}

The constant polynomial $v:=1$ is a {\em highest weight vector of weight $(\al,\zeta)$} in $\B$, which means
$$
a_0v=\al v, \quad zv=\zeta v,\quad a_nv=0 \qquad(n\in\Z_{>0}).
$$
Note that highest weight vector $v$ of weight $(\al,\zeta)$ spans a $1$-dimensional $H_+$-submodule $\C_{\al,\zeta}$, where 
$$H_+:=\spa(z,a_n\mid n\ge 0).$$ 
In view of the PBW Theorem, it is then clear that 
$$
\B(\al,\zeta)=\ind^{U(H)}_{U(H_+)}\C_{\al,\zeta}.
$$

\begin{Lemma} %\label{}%{\rm \cite{}}%{\bf ()}
Let $V$ be a representation of $H$, which admits a nonzero highest weight vector $v$ of weight $(\al,\zeta)$ with $\zeta\neq 0$. Then there is a unique $H$-module homomorphism $\phi$ from $B(\al,\zeta)$ to $V$ such that $\phi(1)=v$. This homomorphism is injective, and the vectors of the form $a_{-1}^{k_1}\dots a_{-n}^{k_n}v$ are linearly independent. If these vectors span $V$, then $V$ is isomorphic to $\B(\al,\zeta)$. In particular, this is the case
if V is irreducible.
\end{Lemma}
\begin{proof}
By the adjointness of tensor and Hom (Frobenius reciprocity), we have an $H$-module homomorphism $\phi$ from $B(\al,\zeta)$ to $V$ defined by 
$$
\phi (f(\dots,x_n,\dots))=f(\dots,\frac{1}{\zeta n}a_{-n},\dots)v.
$$
Since $B(\al,\zeta)$ is irreducible, we have $\ker \phi= 0$.
\end{proof}

Define a grading on $\B$ by setting $$\deg(x_k)=k.$$ Then the dimension of the $j$th graded component $\B_j$ is equal to the number of partitions of $j$, which we denote by $p(j)$:
$$
\dim \B_j=p(j).
$$ 

Let us now return to the fermionic Fock space $\F$. As a module over $H$, it has a highest weight vector $v_\emptyset$ of weight $(0,1)$. 

\begin{Theorem} %\label{}%{\rm \cite{}}%
{\bf (Boson-Fermion Correspondence)}
There is a unique isomorphism 
$$\si:\F\to \B=\B(1,0)$$
of $H$-modules, which maps $v_\emptyset\mapsto 1$. In particular, $\F$ is irreducible as an $H$-module. 
\end{Theorem}
\begin{proof}
By the previous lemma, we have a unique homomorphism $\phi:\B\to \F$ which maps $1$ to $v$, and $\phi$ is necessarily injective. Note also that $\phi$ is homogeneous with respect to the gradings of $\B$ and $\F$. By comparing the dimensions of the graded components of $\B$ and $\F$, we conclude that $\phi$ is an isomorphism. 
\end{proof}

\section{Schur polynomials}\label{S3.2}
 We want to determine the polynomials in $\B$
% $$S_\la:=\si(v_\la)\in\B\qquad(\la\in\mathcal{P}),$$
 which correspond to the natural basis elements $v_\la$ of $\F$ under the boson-fermion correspondence. 
 
The {\em elementary Schur polynomials} $S_k (x)\in \B$ are defined by the generating function
\begin{equation}\label{DSchur}
\sum_{k\in\Z}S_k(x) z^k = \exp\sum_{n\in\Z_{>0}} x_n z^n.
\end{equation}
An easy exercise with multinomial coefficients shows that
$$
S_k(x)=
\left\{
\begin{array}{ll}
0 &\hbox{if $k<0$};\\
1 &\hbox{if $k=0$};\\
\displaystyle\sum_{k_1+2k_2+\dots=k} \frac{x_1^{k_1}}{k_1!} \frac{x_2^{k_2}}{k_2!}\dots
&\hbox{if $k>0$}.
\end{array}
\right.
$$
For example,
\begin{align*}
S_1(x)&= x_1,
\\
 S_2(x) &= \frac{x_1^2}{2}+x_2,
 \\
S_3(x) &= \frac{x_1^3}{6} + x_1x_2 + x_3 ,
\\
S_4(x) &= \frac{x_1^4}{24} + \frac{x_2^2}{2} + \frac{x_1^2x_2}{2}+ x_1x_3 + x_4.
\end{align*}

Now, to each $\la=(\la_1,\dots,\la_n)\in\Par$, we define
\begin{equation}\label{DSLa}
\begin{split}
S_\la(x):=&\det(S_{\la_i+j-i}(x))_{1\leq i,j\leq n}
\\
=&
\left|
\begin{matrix}
S_{\la_1} & S_{\la_1+1} & S_{\la_1+2} & \dots & S_{\la_1+n-1} \\
S_{\la_2-1} & S_{\la_2} & S_{\la_2+1} & \dots & S_{\la_2+n-2} \\
\vdots          & \vdots      &          \vdots & \vdots & \vdots            \\
S_{\la_n+1-n} & S_{\la_n+2-n} & S_{\la_n+3-n} & \dots & S_{\la_n}
\end{matrix}
\right|.
\end{split}
\end{equation}

For example
\begin{align*}
S_{(1,1)}(x)&=
\left|
\begin{matrix}
S_1 & S_2  \\
S_0 & S_1
\end{matrix}
\right|
=
\left|
\begin{matrix}
S_1 & S_2  \\
1 & S_1
\end{matrix}
\right|
=S_1^2-S_2=\frac{x_1^2}{2}-x_2,
\\
S_{(2,1)}(x)&=
\left|
\begin{matrix}
S_2 & S_3  \\
S_0 & S_1
\end{matrix}
\right|
=
S_2S_1-S_3=\frac{x_1^3}{3}-x_3,
\\
S_{(2,2)}(x)&=
\left|
\begin{matrix}
S_2 & S_3  \\
S_1 & S_2
\end{matrix}
\right|
=S_2^2-S_1S_3=\frac{x_1^4}{12}-x_1x_3+x_2.
\end{align*}

It is easy to see that with respect to our grading on $\B$, we have
$$
\deg(S_\la(x))=|\la|.
$$

\begin{Remark} \label{RSF}%{\rm \cite{}}%{\bf ()}
{\rm 
For those who know something about the ring of symmetric functions, here is what is going on: we identify $\B$ with the ring $\La$ of symmetric functions %in the variables $y_1,y_2,\dots$ 
so that $x_k$ corresponds to the $k$th power sum symmetric function 
%$p_k:=y_1^k+y_2^k+\dots$ 
divided by $k$:
$$
\B\leftrightarrow \La,\ x_k\leftrightarrow \frac{p_k}{k}\qquad(k\in\Z_{\geq 0}).
$$
Then 
$$
\sum_{k\in\Z}S_k z^k = \exp\sum_{k\in\Z_{>0}} \frac{p_k}{k} z^k.
$$
This is a well-known expression which defines the elementary symmetric functions $h_k$, see \cite[proof of (2.14)]{M}, so we identify $S_k\leftrightarrow h_k$. Then the Schur polynomial corresponds to the corresponding Schur's symmetric function: $S_\la\leftrightarrow s_\la$, thanks to the Jacobi-Trudi formula, see \cite[(3.4)]{M}. 
}
\end{Remark}

Recall the boson-fermion correspondence $\si:\F\iso \B$. 

\begin{Theorem} %\label{}%{\rm \cite{}}%{\bf ()}
For all $\la\in\Par$, we have 
$$
\si(v_\la)=S_\la.
$$
\end{Theorem}
\begin{proof}
Fix a partition $\la$. 
Under the boson-fermion correspondence, we have  
\begin{equation}\label{E6.10}
\si\big(\exp(y_1\La_1+y_2\La_2+\dots)(v_\la)\big)=
\exp(y_1a_1+y_2a_2+\dots)\si(v_\la). 
\end{equation}
We want to compare the ``$y$-coefficients'' of the highest weight vector in the right and in the left. Let $\si(v_\la)=:P(x)\in\B$. 

In the right hand side of (\ref{E6.10}), we have a bosonic picture, and so the elements $a_k$ for $k>0$ are represented by the operators $\partial/\partial x_k$: 
$$
\exp(y_1a_1+y_2a_2+\dots)=\exp \sum_{j\geq 1}y_j\frac{\partial}{\partial x_j}.
$$
Denote by $F(y)$ the coefficient of $1$ when this operator is applied to $P(x)$. Then, using multivariable Taylor series decomposition, we get 
$$
F(y)=\big(\exp \sum_{j\geq 1}y_j\frac{\partial}{\partial x_j}\big)P(x)|_{x=0}=%P(x+y)|_{x=0}=
P(y). 
$$

Let us turn to the left hand side of (\ref{E6.10}), which is a fermionic picture. We note that that %for a positive $k$ 
the element $\La_k$ can be thought of as the $\Z\times \Z$ matrix with $1$'s on the $k$th diagonal and $0$'s elsewhere, in other words
$$
\La_k=\sum_{n\in\Z}E_{n,n+k}.
$$ 
Note that the product of matrices $\La_k\La_m$ makes sense, and 
$$
\La_k\La_m=\La_k+\La_m\qquad(k,m\in\Z).
$$
In particular, $\La_1^j=\La_j$ for $j$ positive. (This is {\em not} an  equality of operators on $\F$, but rather equality of matrices in a ring of infinite matrices with certain finiteness conditions, for example we might require that all matrices are upper triangular.) 

We can now consider the matrix $\exp(y_1\La_1+y_2\La_2+\dots)$ as an element of the group $U(\C[y_1,y_2,\dots])$ of upper unitriangular $\Z\times \Z$ matrices with coefficients polynomials in $y_1,y_2,\dots$. Moreover,
$$\exp(\sum_{j\geq 1}y_j\La_j)=\exp(\sum_{j\geq 1}y_j\La_1^j)=\sum_{k\geq 0}S_k(y)\La_k,
$$
where we have used the definition (\ref{DSchur}) of the  Schur polynomials $S_k$ for the last equality. In other words, $\exp(\sum_{j\geq 1}y_j\La_j)\in U(\C[y_1,y_2,\dots])$ is the matrix, where all entries on the $k$th diagonal are equal to  $S_k(y)$. 

We can `integrate' the fermionic Fock space $\F$ to make it also a module over the upper unitraiangular group $U(\C)$, so that the action of the matrix $g\in U(\C)$ is defined as usual:
$$
g(v_{i_0}\wedge v_{i_{-1}}\wedge\dots)=g(v_{i_0})\wedge g( v_{i_{-1}})\wedge\dots.
$$

Now, by linear algebra, if $A\in U(\C)$, then the $v_{j_0}\wedge v_{j_{-1}}\wedge\dots$-coefficient of $A(v_{i_0})\wedge A( v_{i_{-1}})\wedge\dots
$
is equal to the minor
$$
\det A(j_0,j_{-1},\dots;i_0,i_{-1},\dots),
$$
where $A(j_0,j_{-1},\dots;i_0,i_{-1},\dots)$ is the submatrix of $A$ obtained by taking the entries that lie in the rows $j_0,j_{-1},\dots$ and the columns $i_0,i_{-1},\dots$. 

The highest weight vector $v_\emptyset=v_0\wedge v_{-1}\wedge\dots$. So the $v_\emptyset$-coefficient of 
$A(v_{i_0}\wedge v_{i_{-1}}\wedge\dots)$ is the minor
$$
\det A(0,{-1},{-2},\dots;i_0,i_{-1},i_{-2},\dots).
$$
For the matrix $A=\exp(\sum_{j\geq 1}y_j\La_j)$ with all entries on the $k$th diagonal being equal to $S_k(y)$, this boils down to
the determinant of the $\Z_{\geq 0}\times\Z_{\geq 0}$ matrix with $(r,s)$ entry equal to $S_{i_r+s}(y)=S_{i_r+r+s-r}(y)$. Recall that $v_\la=v_{\la_1}\wedge v_{\la_2-1}\wedge\dots$. Comparing to (\ref{DSLa}), the $v_\emptyset$-coefficient of $v_\la$ is $S_\la(y)$. The theorem follows since $\si(v_\emptyset)=1$. 
\end{proof}

Now, recall that we have a non-degenerate symmetric contravariant form $(\cdot,\cdot)$ on $\F$ with respect to which the basis $\{v_\la\mid\la\in\Par\}$ is orthonormal. Let $\omega$ be a linear operator on $H$ defined as follows: 
$$
\omega: H\to H,\ a_n\mapsto a_{-n},\ z\mapsto z.
$$
Note that $\omega$ is an involute antiautomorphism of the Lie algebra $H$. Then using the fact that $\F$ is irreducible as an $H$-module, we see that $(\cdot,\cdot)$ is uniquely characterized as a bilinear form on $\F$ such that
$$
(1,1) =1\ \text{and}\ (hv,w)=(v,\omega(h)w)\ \text{for all $h\in H$ and $v,w\in \F$}.
$$
Using boson-fermion correspondence, we may transfer $(\cdot,\cdot)$ to a form on $\B$, which we again denote $(\cdot,\cdot)$. Then the Schur polynomials form an orthonormal basis of $\B$ with respect to the contravariant form $(\cdot,\cdot)$.

We have observed in the end of Section~\ref{S2.2}, using the Murnaghan-Nakayama Rule, that the character value $\chi^\la(c_\rho)$ for a partition $\rho=(\rho_1,\dots,\rho_l)$ can be found as follows:
$$
\chi^\la(c_\rho)=(v_\la,\La_{-\rho_1}\dots\La_{-\rho_l}v_\emptyset).%=(\La_{\rho_1}\dots\La_{\rho_l}v_\la,v_\emptyset).
$$
We can now transfer this to bosons as follows:
\begin{equation}\label{EChOrth}
\chi^\la(c_\rho)=(S_\la,a_{-\rho_1}\dots a_{-\rho_l}1)=(S_\la,\rho_1x_{\rho_1}\dots \rho_lx_{\rho_l})
\end{equation}
Denote $P_j=jx_j$, and $P_\rho:=P_{\rho_1}\dots P_{\rho_l}$. Then the above expression becomes
$$
\chi^\la(c_\rho)=(S_\la,P_{\rho}).
$$
Hence the character values are the change of basis matrix defined from
\begin{equation}\label{EFrB}
P_\rho=\sum_\la \chi^\la(c_\rho) S_\la.
\end{equation}
If we push this back to the ring of symmetric functions using Remark~\ref{RSF}, we get the famous Frobenius formula:
$$
p_\rho=\sum_\la \chi^\la(c_\rho) s_\la.
$$
This was historically the first description of the characters of the symmetric group (obtained in 1899 by Frobenius), see \cite{C}. 
It is not difficult to invert (\ref{EFrB}). Setting 
$$x^\la:=x_{\la_1}\dots x_{\la_l}=x_1^{l_1}x_2^{l_2}\dots$$ for a partition 
$$\la=(\la_1,\dots,\la_l)=(1^{l_1},2^{l_2},\dots),$$ 
we note that 
$$\{x_\la\mid\la\in\Par\}$$ is an orthogonal basis. In fact, we have
$$
(x_\la,x_\mu)=\de_{\la,\mu}l_1!l_2!\dots.
$$
This is proved by induction, see Exercise~\ref{ExXOrth}.
Denote $Z_\la:=l_1!l_2!$. 
Then (\ref{EChOrth}) implies
$$
S_\la=\sum_\rho\frac{1}{Z_\rho}\chi^\la(c_\rho)x^\rho.
$$
So after harmless normalization the coefficients of the polynomial $S_\la$ are simply character values. 

I might handwave some exciting connection to soliton equations here ... details can be found in \cite{KR}.

\chapter{Day Four}

We'll see how much time we have left. At the moment I plan to tell you something fun about Khovanov-Lauda-Rouquier algebras and their relevance for symmetric groups, various Hecke algebras, and other areas.

\chapter{Exercises}

\section{Exercises used in the lectures}

\begin{Exercise} \label{ExRepAlg}%{\rm \cite{}}%
{\bf (Representations vs. Modules)}
{\rm\small 
Let $A$ be an algebra over a field $F$.

(a) A {\em representation} of $A$ means 
a pair $(V, \rho)$ where $V$ is a vector space and 
$\rho:A \rightarrow \End_F(V)$ is an algebra homomorphism.
A morphism $f:(V,\rho) \rightarrow (W,\sigma)$ between two
representations of $A$
means a linear map $f:V \rightarrow W$ such that
$f \circ \rho(a) = \sigma(a) \circ f$ for all $a \in A$.
This defines the category $\Rep{A}$ 
of all representations
of $A$.
Prove that the category $\Rep{A}$ is
{\em isomorphic} to the category $\Mod{A}$.

(b) A {\em matrix representation} of $A$ means
a ring homomorphism $\rho:A \rightarrow M_n(F)$
for some $n \geq 0$. Morphisms of matrix representations
are defined similarly to (i).
This defines the category $\MatRep{A}$ 
of matrix representations of $A$. Prove that the
category $\MatRep{A}$ is {\em equivalent} to the
category of all finite dimensional $A$-modules.
Could we replace the word ``equivalent'' with ``isomorphic'' here?
}
\end{Exercise}

\begin{Exercise} \label{ExMasch}%{\rm \cite{}}%{\bf ()}
{\rm\small 
{\bf (Maschke's Theorem)}
Let $G$ be a finite group and $F$ be a field of characteristic $p\geq 0$. Then every $FG$-module is semisimple if and only if $p$ does not divide the order of the group $|G|$. 

Prove this in steps as follows:

(i) If $p$ divides $|G|$, consider the $1$-dimensional submodule of the left regular module ${}_{FG} FG$ spanned by the element $x:=\sum_{g\in G} g$, and show that this submodule is not a direct summand of the regular module. %Indeed,  if $C$ is a complement to $Fx$ in ${}_{FG}FG$ and  $\sum_{h\in G} a_hh\in C$, then $x(\sum_{h\in G} a_hh)=(\sum_{h\in G}a_h)x\in C$.)

(ii) Let $p\nd|G|$, and $W\subseteq V$ be left $FG$-modules. We need to show that there is a submodule $X\subseteq V$ with $V=W\oplus X$. 
Let $Y$ be an $F$-subspace of $V$ with $V=W\oplus Y$. The projection $\pi:V\to W$ along $Y$ is a linear transformation. We define a map $\phi:V\to V$ by the formula
$$
\phi(v)=\frac{1}{|G|}\sum_{g\in G} g^{-1}\pi(gv)\qquad(v\in V).
$$
Check that $\phi$ is  an $FG$-module homomorphism. Check that  $\im \phi = W$, and so $V=W\oplus\ker \phi$ as modules.
}
\end{Exercise}

\begin{Exercise} \label{ExHopfAlg}%{\rm \cite{}}%{\bf ()}
{\rm\small {\bf (Inner Tensor Product of Modules and Hopf Algebras)} 
Let $A$ be an associative $F$-algebra with multiplication map $m:A\otimes A\to A$, and let $\iota:F\to A,\ c\mapsto c1_A$. 

(a) If there is a homomorphism of algebras $\De:A\to A\otimes A$ (called {\em comultiplication}) we can define the structure of an $A$-module on $V\otimes W$ via $a x=\De(a)x$ for all $a\in A$ and $x\in V\otimes W$. 

(b) Comultiplication is {\em coassociative} if $(\id_A\otimes\De)\circ \De=(\De\otimes\id_A)\circ \De$. If comultiplication is coassociative then for any $A$-modules $X,U,V$, the isomorphism of vector spaces 
$$
(X\otimes U)\otimes V\iso X\otimes(U\otimes V),\ (x\otimes u)\otimes v\mapsto x\otimes (u\otimes v)
$$
is an isomorphism of $A$-modules. 

(c) An $F$-algebra homomorphism $\eps: A\to F$ defines a structure of an $A$-module on $F$. 
The homomorphism $\eps$
is a {\em counit} if $(\eps\bar\otimes\id_A)\circ \De=\id_A=(\id_A\bar\otimes\eps)\circ \De$, where $\bar\otimes $ means that one should use the natural isomorphisms $F\otimes A\iso A$ and $A\otimes F\iso A$.  
If $\eps$ is a counit then the natural isomorphisms of vector spaces 
$F\otimes V\cong V\cong V\otimes F$ 
are isomorphisms of $A$-modules.

(d) An $F$-algebra antiautomorphism $\si:A\to A$ is an {\em antipode} if 
$$m\circ(\id_A\otimes \si)\circ \De=\iota\circ \eps=m\circ( \si\otimes\id_A)\circ \De.$$
Given an $A$-module $V$, we can use $\si$ to define the structure of an $A$-module on the dual vector space $V^*$ as follows: $(af)(v):=f(\si(a)v)$ for all $a\in A,f\in V^*, v\in V$. Use the assumption that $\si$ is an anti-homomorphism to verify that this makes $V^*$ into an $A$-module. Use the assumption that $\si$ is an antipode to verify that the natural maps
$$V^*\otimes V\to F,\ f\otimes v\mapsto f(v),\qquad V\otimes V^*\to F,\  v\otimes f\mapsto f(v)
$$
are $A$-module homomorphisms.

(e) An associative algebra $A$ with coassociative comultiplication $\De$, counit $\eps$ and antipode $\si$ is called a {\em Hopf algebra}. A Hopf algebra is {\em cocommutative} if $\si\circ \De=\De$ where $\si$ is the linear map $A\otimes A\mapsto A\otimes A, a\otimes b\mapsto b\otimes a$. If $A$ is cocommutative then the natural isomorphism
$
V\otimes W\iso W\otimes V,\ v\otimes w\mapsto w\otimes v
$
is an isomorphism of $A$-modules. 

(f) If $G$ is a group, define linear maps $\De:g\mapsto g\otimes g$, $\eps:g\mapsto 1$, and $\si:g\mapsto g^{-1}$ via their action on the group elements and extending to $FG$. These yield a structure of a cocommutative Hopf algebra on the group algebra $FG$. 
}
\end{Exercise}

\begin{Exercise} \label{ECharDualTensProd}%{\rm \cite{}}%
{\bf (Character of a Dual Module)}
{\rm\small 
$\chi_{V^*} = \overline{\chi_V}$ ({\em Hint:} diagonalize $g\in G$ on $V$, note that diagonal entries are roots of unity, and use that $\eps^{-1}=\bar\eps$ for a root of unity $\eps$ to conclude that $\chi(g^{-1})=\overline{\chi(g)}$ ). 
}
\end{Exercise}

\begin{Exercise} \label{ExScPr}%{\rm \cite{}}%
{\bf (Existence of Invariant Inner Products)}
{\rm\small 
Let $G$ be a finite group, and $V$ be a $\C G$-module. 

(i) Prove that there exists a (non-degenerate) $G$-invariant inner product on $V$. ({\em Hint:} for the proof of existence start with any inner product and use ``averaging over $G$".)

(ii) Prove that if $V$  is irreducible, then the inner product is unique up to scalar. 

(iii) If $(\cdot,\cdot)$ is a $G$-invariant inner product on $V$, then $(e_iv,w)=(v, e_iw)$, where $e_1,\dots,e_r$ are the idempotents defined in (\ref{EIdG}). 

(iv) If $W_1\not\cong W_2$ are two non-isomorphic irreducible submodules of $V$, then $W_1\perp W_2$ with respect to any $G$-invariant inner product on $V$. 
}
\end{Exercise}

\begin{Exercise}\label{ExCanTens}
{\rm \small 
Let $V$ be a free right $R$-module with basis $\{v_i\mid i\in I\}$ and $W$ be a left $R$-module. For each $i\in I$, denote by $v_i\otimes W$ the abelian subgroup of $V\otimes_R W$ 
%generated by 
consisting of all pure tensors of the form $v_i\otimes w$ with $w\in W$. 

(a) For each $i\in I$, there is an isomorphism of abelian groups $$W\to v_i\otimes W,\ w\mapsto v_i\otimes w.$$ % is an isomorphism of abelian groups. 

(b) $V\otimes_R W=\bigoplus_{i\in I} v_i\otimes W$ as abelian groups.
}
\end{Exercise}

\begin{Exercise} \label{ExZeroChar} %{\rm \cite{}}%{\bf ()}
{\rm\small 
If $H$ is a subgroup of a finite group $G$, $g\in G$ is not conjugate to an element of $H$, and $V$ is a $\C G$-module induced from $H$, then the character of $V$ on $g$ is zero
}
\end{Exercise}

\begin{Exercise} \label{ExCan}%{\rm \cite{}}%
{\bf (Uniqueness of isotypic components)}
{\rm\small 
Let $A$ be an algebra and $V$ be an $A$-module. Let $V= L_1\oplus\dots\oplus L_r$ be a module decomposition with $L_1,\dots, L_r$ being pairwise non-isomorphic simple $A$-modules. Assume that $V= L_1'\oplus\dots\oplus L_r'$ is another module decomposition of $V$ such that $L_1'\cong L_1,\dots, L_r'\cong L_r$. Then $L_1'=L_1,\dots, L_r'\cong L_r$. 
}
\end{Exercise}

\begin{Exercise} \label{ExConjCl}%{\rm \cite{}}%{\bf ()}
{\rm\small 
Describe conjugacy classes of $S_n$. Show that the number of these conjugacy classes is equal to the number of partitions of $n$. 
}
\end{Exercise}

\begin{Exercise} %\label{}%{\rm \cite{}}%
{\bf (Olshanskii's Lemma)}
{\rm\small 
Fill in details in the proof of Olshanskii's Lemma. 
}
\end{Exercise}

\begin{Exercise} \label{ExW'}%{\rm \cite{}}%
{\bf (Gelfand-Zetlin Spectrum and Standard Tableaux)}
{\rm\small 
Prove Lemma~\ref{L010703_8}.  
}
\end{Exercise}

\begin{Exercise}%\label{}%{\rm \cite{}}%
{\bf (Length Function on symmetric group)}
{\rm\small 
A simple transposition is a transposition of the form $(m,m+1)$. For $w\in S_n$, define the length of $w$, written $\ell(w)$ to be the minimal number $r$ such that $w$ can be written as a product of $r$ simple transpositions. Then $\ell(w)$ is equal to the number of inversions in the sequence $(w(1),\dots,w(n))$, i.e. 
$$\ell(w)=\{(i,j)\mid 1\leq i<j\leq n\ \text{and}\ w(i)>w(j)\}.$$
}
\end{Exercise}

\begin{Exercise} \label{ExXOrth}%{\rm \cite{}}%
{\bf (Orthogonality of Monomial Basis in Bosonic Fock Space)}
{\rm\small 
Let $(\cdot,\cdot)$ be the contravariant form on the bosonic Fock space $\B$ with respect to the action of the Heisenberg algebra normalized so that $(1,1)=1$. 
For a partition 
$$\la=(1^{l_1},2^{l_2},\dots),$$
set
$x^\la:=x_1^{l_1}x_2^{l_2}\dots$. 
 Use induction to show that $
(x_\la,x_\mu)=\de_{\la,\mu}l_1!l_2!\dots.
$
}
\end{Exercise}

\begin{Exercise} %\label{}%{\rm \cite{}}%
{\bf (Hook Formula and Standard Tableaux)}
{\rm\small 
Compute the dimension of the irreducible representation corresponding to the partition $\la=(5,3,3,1)$ by two methods:

(i) using Hook Formula;

(ii) by writing a program which counts the number of standard $\la$-tableaux. 
}
\end{Exercise}

\begin{Exercise} \label{ExSF1}%{\rm \cite{}}%
{\bf (Basic Symmetric Functions)}
{\rm\small 
Let $\La$ be the ring of symmetric functions in infinitely many variables $x_1,x_2,\dots$; this is the inverse limit of the rings $\La_n$ of the symmetric functions in $n$ variables with respect to the maps $\La_m\to \La_{m-1}$ which put the last variable $x_m$ to zero.  

For a composition $\al=(\al_1,\dots,\al_n)\in\Z_{\geq 0}^n$ define $x^\al:=x_1^{\al_1}\dots x_n^{\al_n}$. The monomials $x^\al$ and $x^\be$ are equivalent if they can be obtained from each other by permuting variables. 

For a partition $\la=(\la_1,\la_2,\dots,\la_l)$, 
{\em monomial symmetric function}
$$
m_\la:=\sum x^\al,
$$
where the sum is over all {\em distinct} monomials $x^\al$ equivalent to $x^\la$. For example, 
$$
m_{(2,1)}=x_1^2x_2+x_1x_2^2+x_1^2x_3+x_3x_1^2+x_2^2x_3+x_3x_2^2+\dots
$$

(i) Prove that $\{m_\la\mid\la\in\Par\}$ is a basis of $\La$.

(ii) For $r\geq 0$, define {\em elementary symmetric functions} $$e_r:=m_{(1^r)}=\sum_{i_1<\dots<i_r}x_{i_1}\dots x_{i_r},$$ 
and set $E(t):=\sum_{r\geq 0}e_rt^{r}$. Prove that $E(t)=\prod_{i\geq 1}(1+x_it).$

(iii) For $r\geq 0$, define {\em complete symmetric functions} $h_r:=\sum_{|\la|=r}m_\la$, and set $H(t):=\sum_{r\geq 0}h_rt^r$. Prove that $H(t)=\prod_{i\geq 1}(1-x_it)^{-1}.$

(iv) For $r\geq 1$, define {\em power sum symmetric functions} $p_r:=m_{(r)}=x_1^r+x_2^r+\dots$, and set $P(t):=\sum_{r\geq 1}p_rt^{r-1}$. Prove that $P(t)=\frac{d}{dt}\log H(t)=H'(t)/H(t).$

(v) Working with $n$ variables $x_1,\dots,x_n$, for $1\leq k\leq n$, define $$e_r^{(k)}:=e_r(x_1,\dots,x_{k-1},x_{k+1},\dots,x_n).$$
%and let $M$ denote the $n\times n$ matrix $M:=\big((-1)^{n-i}e_{n-i}^{(k)}\big)_{1\leq k,i\leq n}$. 
Note that 
$$
E^{(l)}(t):=\sum_{r=0}^{n-1}e_r^{(k)}t^r=\prod_{i\neq k}(1+x_it).
$$
Deduce that $H(t)E^{(k)}(-t)=(1-x_kt)^{-1}$. By comparing $t^a$ coefficient on both sides conclude that
$$
\sum_{j=1}^nh_{a-n+j}\cdot(-1)^{n-j}e_{n-j}^{(k)}=x_k^a\qquad(a\in \Z_{\geq 0}).
$$
Deduce for any composition $\al=(\al_1,\dots,\al_n)$ that $A_\al=H_\al M$, where
$A_\al:=(x_j^{\al_i})_{1\leq i,j\leq n}$, $H_\al:=(h_{\al_i-n+j})_{1\leq i,j\leq n}$, and $M:=\big((-1)^{n-i}e_{n-i}^{(k)}\big)_{1\leq k,i\leq n}$.
}
\end{Exercise}

\begin{Exercise} %\label{}%{\rm \cite{}}%
{\bf (Schur Functions)}
{\rm\small 
Let us first work with a finite number of variables $x_1,\dots,x_n$. For a partition $\mu=(\mu_1,\dots,\mu_n)$ with at most $n$ parts, define the polynomial $a_\mu=\det(x_i^{\mu_j})_{1\leq i,j\leq n}$. 
Consider the special partition $\de:=(n-1,n-2,\dots,1,0)$ and the  polynomial $a_{\la+\de}$. 

(i) Prove that $a_{\la+\de}$ is non-zero and {\em skew-symmetric}, i.e. a permutation of $x_i$ and $x_j$ for $i\neq j$ yields the polynomial equal to the negative of $a_{\la+\de}$. Deduce that $a_{\la+\de}$ is divisible by all $x_i-x_j$ for $i\neq j$ in the polynomial ring. Deduce that $a_{\la+\de}$ is divisible by $\prod_{1\leq i<j\leq j}(x_i-x_j)=a_\de$ in the polynomial ring. Define the polynomial
$$
s_\la=s_\la(x_1,\dots,x_n):=\frac{a_{\la+\de}}{a_\de}.
$$
Prove that $s_\la$ is symmetric. 

(ii) Prove that $s_\la(x_1,\dots,x_m,0)=s_\la(x_1,\dots,x_n)$, and deduce that there is a well-defined Schur's function in infinitely many variables obtained as the inverse limit of the Schur's functions in finitely many elements. 

(iii) {\bf (Jacobi-Trudi Formula)} Let $\la=(\la_1,\dots,\la_n)$ be a partition with at most $n$ non-zero parts. Apply determinants to the both sides of the equation $A_\al=H_\al M$ obtained in Exercise~\ref{ExSF1} for the case where $\al=\la+\de$ and interpret the equality as 
$s_\la=\det(h_{\la_i-i+j})_{1\leq i,j\leq n}$.
}
\end{Exercise}

\section{Other exercises}

\begin{Exercise} %\label{}%{\rm \cite{}}%{\bf ()}
{\rm\small 
For the left regular module ${}_{\C G}\C G$ of the group algebra we have 
%decomposes as a direct sum of irreducible representations as follows:
$
{}_{\C G}\C G\cong L_1^{\oplus n_1}\oplus\dots\oplus L_r^{\oplus n_r}.
$
}
\end{Exercise}

\begin{Exercise} %\label{}%{\rm \cite{}}%{\bf ()}
{\rm\small 
Let $V_n=\C\{1,2,\dots,n\}$ be the natural $n$-dimensional permutation module over $S_n$. Prove that $V_n\cong \ind_{S_{n-1}}^{S_n}\triv_{S_{n-1}}$ and 
$\res^{S_n}_{S_{n-1}}V_n\cong V_{n-1}\oplus\triv_{S_{n-1}}$. 
}
\end{Exercise}

\begin{Exercise} %\label{}%{\rm \cite{}}%{\bf ()}
{\rm\small 
True or false? There exists a finite group $G$ with precisely four inequivalent irreducible representations of dimensions $1,2,3$ and $4$. 
}
\end{Exercise}
\iffalse{
False: $4$ does not divide $30$.
}\fi

\begin{Exercise} %\label{}%{\rm \cite{}}%{\bf ()}
{\rm\small 
A $\C G$-module $V$ is irreducible if and only if $V^*$ is irreducible as a $\C G$-module. 
}
\end{Exercise}

\begin{Exercise} %\label{}%{\rm \cite{}}%{\bf ()}
{\rm\small 
Let $V$ and $W$ be finite dimensional $\C G$-modules. Define a $\C G$-module structure on $\Hom_\C(V,W)$ so that $\Hom_\C(V,W)\cong V^*\otimes W$ as $\C G$-modules and $\Hom_{\C G}(V,W)= \Hom_\C(V,W)^G$, the subspace of $G$-invariants. 
}
\end{Exercise}

\begin{Exercise} %\label{}%{\rm \cite{}}%{\bf ()}
{\rm\small 
 It is known from Feit-Thompson's theorem that a non-abelian simple group has an even order. Use this fact  to prove that no simple group has an irreducible complex representation of  degree $2$. ({\em Hint:} Use $\det: GL_2(\C)\to\C^\times$ to show that the image of the simple group is contained in $SL_2(\C)$. Then think about image of an element of order $2$.)
}
\end{Exercise}
\iffalse{
element of order 2 must go to the scalar matrix -I, which shows that the center of $G$ is non-trivial
}\fi

\begin{Exercise} %\label{}%{\rm \cite{}}%{\bf ()}
{\rm\small 
True or false? A non-abelian group of order $55$ has exactly five one-dimensional complex representations up to isomorphism.
}
\end{Exercise}
\iffalse{
True, as $|G/G'|=5$. 
}\fi

\begin{Exercise} %\label{}%{\rm \cite{}}%{\bf ()}
{\rm\small 
Let $G$ be a finite group such that every irreducible $\C G$-module is one-dimensional. Then $G$ is abelian. 
}
\end{Exercise}
\iffalse{
In regular representation every element acts with a diagonal matrix. 
}\fi

\begin{Exercise} %\label{}%{\rm \cite{}}%
{\rm\small {\bf (Irreducible representations of dihedral groups)}
Let $D_{2n}=\lan a,b\mid a^n=b^2=1,\ bab^{-1}=a^{-1}\ran$ be the dihedral group, $\eps:=e^{2\pi i/n}$, and  set 
$$
B:=
\left(
\begin{matrix}
0 & 1  \\
1 & 0
\end{matrix}
\right), \quad 
A_j:=\left(
\begin{matrix}
\eps^j & 0  \\
0 & \eps^{-j}
\end{matrix}
\right)\qquad(1\leq j<n).
$$

(a) Show that for $1\leq j<n$ there is a matrix representation $\rho_j:D_{2n}\to GL_2(\C)$ such that $\rho(a)=A_j$ and $\rho(b)=B$. 

(b) Use Schur's Lemma to prove that $\rho_1,\dots,\rho_{n-1}$ are irreducible unless $n$ is even and $j=n/2$.

(c) Use Schur's Lemma to prove that the representations $\rho_1,\dots,\rho_{\lfloor (n-1)/2\rfloor}$ are pairwise non-isomorphic. 

(d) If $n=2k$ is even, then  $D_{2n}$ has four non-isomorphic one-dimensional representations, which together with $\rho_1,\dots,\rho_{k-1}$ give a complete and irredundant list of irreducible $\C D_{2n}$-modules up to isomorphism.

(e) If $n=2k+1$ is odd, then  $D_{2n}$ has two non-isomorphic one-dimensional representations, which together with $\rho_1,\dots,\rho_{k}$ give a complete and irredundant list of irreducible $\C D_{2n}$-modules up to isomorphism.
}
\end{Exercise}

\iffalse{
\begin{Exercise} %\label{}%{\rm \cite{}}%{\bf ()}
{\rm\small 
Let $A:=
\left(
\begin{matrix}
0 & -1  \\
1 & 0
\end{matrix}
\right)$ abd $B:=
\left(
\begin{matrix}
1 & 0  \\
0 & -1
\end{matrix}
\right)
$, and $D_8=\lan a,b\mid a^4=b^2=1,\ bab^{-1}=a^{-1}\ran$ be the dihedral group. 
There exists an irreducible matrix representation $\rho:D_8\to GL_2(\C)$ with $\rho(a)=A$ and $\rho(b)=B$. %This representation is irreducible.  
}
\end{Exercise}
}\fi

\begin{Exercise} %\label{}%{\rm \cite{}}%{\bf ()}
{\rm\small 
Let $G$ be a finite group and $H\leq G$ be a subgroup. Define the functor $\coind_H^G:
\mod{\C H}\to\mod{\C G}$ using the $\Hom_H$-functor instead of the $\otimes_{\C H}$ functor.

(a) Let $g_1,\dots,g_m$ be the left coset representatives of $H$ in $G$ and $V$ be a $\C H$-module. Define the map 
$$
\al_V:\Hom_{\C H}(\C G,V)=\coind_H^G V\to\ind_H^G V=\C G\otimes_{\C H} V,\ \phi\mapsto\sum_{i=1}^m g_i\otimes\phi(g_i^{-1}).
$$
Then $\al_V$ is independent of the choice of the left coset representatives. 

(b) $\al_V$ is a isomorphism of $\C G$-modules.

(c) $\al_V$ defines an isomorphism of the functors $\ind_H^G\cong\coind_H^G$. 

(d) $\ind_H^G$ is both left and right adjoint to $\res_H^G$. 

(e) $\ind_H^G$ is exact, i.e. maps short exact sequences to short exact sequences. 
}
\end{Exercise}

\begin{Exercise} %\label{}%{\rm \cite{}}%{\bf ()}
{\rm\small 
(a) The functor $\ind_H^G$ is additive.  

(b) If $K\leq H\leq G$ then $\ind_H^G\circ \ind_K^H\cong\ind_K^G$. 

(c) $\ind_H^G(V^*)\cong(\ind_H^G V)^*$ for any $V\in\mod{\C H}$. 
}
\end{Exercise}

%The following is an extremely useful exercise:

\begin{Exercise} %\label{}%{\rm \cite{}}%
{\rm\small {\bf (Tensor Identity)}
Let $G$ be a finite group, $H\leq G$, $V\in\mod{\C H}$ and $W\in\mod{\C G}$. Then there is a functorial isomorphism of $\C G$-modules
$$
(\ind_H^G V)\otimes W\cong \ind_H^G(V\otimes \res_H^G W).
$$
}
\end{Exercise}

\begin{Exercise} %\label{}%{\rm \cite{}}%{\bf ()}
{\rm\small 
If $G$ acts transitively on a set $X$ with a point stabilizer $H$, then the permutation module $\C X$ is isomorphic to the induced module $\ind_H^G \triv_H$. 
}
\end{Exercise}

\begin{Exercise} %\label{}%{\rm \cite{}}%{\bf ()}
{\rm\small 
Let $G$ be a finite group, $H\leq G$, and $V\subseteq{}_{\C H}\C H$ be a submodule of the regular module for $H$. As $\C H$ is naturally embedded into $\C G$, we can consider $V$ as a subspace of $\C G$. Then the submodule of ${}_{\C G}\C G$ generated by $V$ is isomorphic to $\ind_H^G V$. 
}
\end{Exercise}

\begin{Exercise} %\label{}%{\rm \cite{}}%{\bf ()}
{\rm\small 
Let $G$ be a finite group and $H\leq G$. Then each irreducible irreducible $\C G$-module is a summand of a module induced from an irreducible $\C H$-module. 
}
\end{Exercise}

\begin{Exercise} \label{ExGrAlDirProd}%{\rm \cite{}}%{\bf ()}
{\rm\small 
If $G$ and $H$ are two groups then we have the following isomorphism of group algebras: $F[G\times H]\cong FG\otimes FH$. 
}
\end{Exercise}

\begin{Exercise} \label{TensProdMatrixAlg}%{\rm \cite{}}%{\bf ()}
{\rm\small 
$M_n(F)\otimes M_m(F)\cong M_{mn}(F)$.
}
\end{Exercise}

\begin{Exercise} \label{ExOuterTensProdAlg}%{\rm \cite{}}%{\bf ()}
{\rm\small {\bf (Outer tensor product of modules)}
Let $A$ and $B$ be $F$-algebras, $V$ be an $A$-module and $W$ be a $B$-module, then $V\otimes W$ has a structure of a module over the algebra $A\otimes B$ such that $(a\otimes b)(v\otimes w)=(av)\otimes(bw)$ for all $a\in A,b\in B,v\in V,w\in W$. This tensor product is sometimes referred to as the {\em outer tensor product} and denoted $V\boxtimes W$. 
}
\end{Exercise}

\begin{Exercise} \label{LHomOuterTens}%{\rm \cite{}}%{\bf ()}
{\rm\small
Let $A$ and $B$ be associative algebras over a field $F$, $V,V'$ be finite dimensional $A$-modules, and $W,W'$ be finite dimensional $B$-modules. Then 
$$
\Hom_{A\otimes B}(V\boxtimes W,V'\boxtimes W')\cong\Hom_A(V,V')\otimes\Hom_B(W,W'). 
$$

For the proof proceed in steps as follows: 

(i) Show that there is an embedding 
$$\Hom_A(V,V')\otimes\Hom_B(W,W')\into  \Hom_{A\otimes B}(V\boxtimes W,V'\boxtimes W')
$$
which maps pure tensor $f\otimes g$ on the left to the map $f\otimes g$

(ii) Note that every element $\phi$ of the tensor product in the left hand side can be written as a finite sum $\sum_i\al_i\otimes \be_i$ for some linear maps $\al_i:V\to V'$ and some {\em linearly independent} linear maps $\be_i:W\to W'$. Using linear independence of the $\be_i$, conclude that each $\al_i$ must belong to $\Hom_A(V,V')$. Now, we can rewrite our expression for $\phi$ as $\sum_j\ga_j\otimes\de_j$ where $\ga_j\in\Hom_A(V,V')$ are linearly independent and $\de_j:W\to W'$ are some linear maps. Now shows that all $\de_j$ are $B$-homomorphisms. 
}
\end{Exercise}

\begin{Exercise}  \label{TOuterTensAlg}%{\rm \cite{}}%{\bf ()}
{\rm\small
Let $A$ and $B$ be finite dimensional associative algebras over an algebraically closed field $F$.
\begin{enumerate}
\item[{\rm (i)}] If $V$ is an irreducible $A$-module and $W$ is an irreducible $B$-module then $V\boxtimes W$ is an irreducible $A\otimes B$-module.
\item[{\rm (ii)}] Every irreducible $A\otimes B$-module is of the form $V\boxtimes W$ for some irreducible $A$-module $V$ and some irreducible $B$-module $W$. 
\end{enumerate}

Prove (i) as follows: let $J(A)$ denote the Jacobson radical of $A$. The module $V\boxtimes W$ factors through to give a module over the quotient $A\otimes B/(A\otimes J(B)+J(A)\otimes B)$, which is semisimple. So it suffices to prove that $\End_{A\otimes B}(V\boxtimes W)=F$, which follows from Exercise~\ref{LHomOuterTens} and Schur's Lemma. 

Prove (ii) as follows: we may assume that $A$ and $B$ are semisimple. Now the result follows from (i) and Exercise~\ref{TensProdMatrixAlg} by counting irreducibles. 
}
\end{Exercise}

\begin{Exercise}  \label{LSymPolMurCen}%{\rm \cite{}}%{\bf ()}
{\rm\small 
For any symmetric polynomial $f$ in $n$ variables, the element $f(x_1,\dots,x_n)$ is central in the group algebra $ R_n$. 
 
To prove this, proceed in steps as follows:

(i) It suffices to check that $f(x_1,\dots,x_n)$ commutes with an arbitrary $s_a$. Write $f(x_1,\dots,x_n)$ as a linear combination of terms of the form $g(x_a,x_{a+1})M$, where $M$ is a monomial not involving $x_a$ and $x_{a+1}$ and $g$ is a symmetric polynomial in two variables. Check that $s_a$ commutes with $M$, so it suffices to check that $s_r$ commutes with $g(x_a,x_{a+1})$. 

(ii) Observe that  $g(x_a,x_{a+1})$ is a polynomial in $x_a+x_{a+1}$ and $x_ax_{a+1}$, and then check that $s_a$ commutes  with these.
}
\end{Exercise}

\begin{Exercise} %\label{}%{\rm \cite{}}%{\bf ()}
{\rm\small 
Let $G$ be a finite group, $X$ be a $G$-set, and $\chi$ be the character of the permutation module $\C X$.  

(a) $(\chi, 1)$ is the number of orbits of $G$ on $X$. In particular, if $G$ is transitive, $\C X$ can be decomposed as $\triv_G\oplus V$ where $V$ does not contain the trivial representation. %Let $\psi$ be the character of $V$. 

(b) If $Y$ is a $G$-set then $G$ acts on $X\times Y$ via $g(x,y)=(gx,gy)$ for $g\in G,x\in X,y\in Y$. Show that $\C[X\times Y]\cong \C X\otimes \C Y$. Deduce that the character of $\C[X\times X]$ is $\chi^2$.   

(c) {\bf\boldmath ($2$-transitivity Criterion)} The following are equivalent: 
\begin{enumerate}
\item[{\rm (i)}] $G$ is $2$-transitive on $X$;
\item[{\rm (ii)}] $G$ has exactly two orbits on $X\times X$;
\item[{\rm (iii)}] $(\chi^2,1)=2$;
\item[{\rm (iv)}] $V$ is irreducible. 
\end{enumerate}
}
\end{Exercise}
\iffalse{(This is the orbit counting lemma: the number of orbits of $G$ on $X$ is the average number of 
fixed points , that is the definition of the inner product on $C(G)$!).
}\fi

\end{document}